\documentclass{article}

\usepackage{amsthm} 
\usepackage{color}

\usepackage{graphicx,epsfig}
\usepackage{epstopdf}
\usepackage{amsmath, amssymb, latexsym, euscript,amscd}
\usepackage{url}
\usepackage[all]{xy}
\usepackage{psfrag}
\usepackage{mathrsfs}

\definecolor{darkgreen}{rgb}{0.0, 0.5, 0.0}

\newtheorem{theorem}{Theorem}[section]
\newtheorem{lemma}[theorem]{Lemma}
\newtheorem{corollary}[theorem]{Corollary} 
\newtheorem{definition}[theorem]{Definition} 
\newtheorem{proposition}[theorem]{Proposition}
\newtheorem{remark}[theorem]{Remark}

\makeatletter
\newtheorem*{rep@theorem}{\rep@title}
\newcommand{\newreptheorem}[2]{%
\newenvironment{rep#1}[1]{%
 \def\rep@title{#2 \ref{##1}}%
 \begin{rep@theorem}}%
 {\end{rep@theorem}}}
\makeatother

\newreptheorem{theorem}{Theorem}
\newreptheorem{lemma}{Lemma}

\def\smallskip{\vspace{.15cm}}
\def\medskip{\vspace{.3cm}}
\def\text{\mbox}
\def\rh2{{\mathbb R}{\mathbb H}^2}
\def\ch2{{\mathbb C}{\mathbb H}^2}
\def\RR{{\mathbb R}}
\def\KK{{\mathbb K}}
\def\CC{{\mathbb C}}
\def\EE{{\mathbb E}}

\def\HH{{\mathbb H}}
\def \XX{{\mathsf X}}

\def\RP{{\mathbb{RP}}}

\def\SL{\operatorname{SL}}
\def\PSL{\operatorname{PSL}}
\def\SO{\operatorname{SO}}

\def\H2R{{\mathbb H}^2\times {\mathbb R}}
\def\eps{\varepsilon}
\def\AdS{\operatorname{AdS}}
\def\dS{\operatorname{dS}}
\def\PSO{\operatorname{PSO}}
\def\PO{\operatorname{PO}}
\def\PU{\operatorname{PU}}
\def\Isom{\operatorname{Isom}}

\def\GL{\mathrm{GL}}
\def\eps{\varepsilon}

\newcommand{\OO}{\operatorname{O}}
\newcommand{\PP}{\operatorname{P}}
\newcommand{\pfbeta}{{\boldsymbol \beta}}
\newcommand{\frakp}{\mathfrak{p}}
\newcommand{\frakq}{\mathfrak{q}}
\newcommand{\fraka}{\mathfrak{a}}
\newcommand{\frakg}{\mathfrak{g}}
\newcommand{\frakh}{\mathfrak{h}}
\newcommand{\frakk}{\mathfrak{k}}

\newcommand{\frakb}{\mathfrak{b}}

\newcommand{\frakl}{\mathfrak{l}}
\newcommand{\frakgl}{\mathfrak{gl}}
\newcommand{\frakn}{\mathfrak{n}}
\newcommand{\frakso}{\mathfrak{so}}
\newcommand{\flag}{\mathcal F}
\newcommand{\PGL}{\operatorname{PGL}}
\newcommand{\Ad}{\mathrm{Ad}}
\newcommand{\ad}{\mathrm{ad}}

\newcommand{\Grp}{\mathsf{Grp}_0}

\newcommand{\Sp}{\operatorname{Sp}}
\newcommand{\charpoly}{\operatorname{char}}
\newcommand{\charpolys}{\operatorname{Char}}

\newcommand{\bv}{\left(\begin{array}{c}}
\newcommand{\ev}{\end{array}\right)}
\newcommand{\bpmat}{\begin{pmatrix}} 
\newcommand{\epmat}{\end{pmatrix}}
\newcommand{\bmat}{\begin{matrix}} 
\newcommand{\emat}{\end{matrix}}

\title{Limits of Geometries}
\author{Daryl Cooper, Jeffrey Danciger, and Anna Wienhard}

\begin{document}
\maketitle

\begin{abstract}
A geometric transition is a continuous path of geometric structures that changes type, meaning that the model geometry, i.e. the homogeneous space on which the structures are modeled, abruptly changes.
In order to rigorously study transitions, one must define a notion of geometric limit at the level of homogeneous spaces, describing the basic process by which one homogeneous geometry may transform into another.
We develop a general framework to describe transitions in the context that both geometries involved are represented as sub-geometries of a larger ambient geometry. Specializing to the setting of real projective geometry, we classify the geometric limits of any sub-geometry whose structure group is a symmetric subgroup of the projective general linear group.
As an application, we classify all limits of three-dimensional hyperbolic geometry inside of projective geometry, finding Euclidean, Nil, and Sol geometry among the limits. We prove, however, that the other Thurston geometries, in particular $\HH^2 \times \RR$ and $\widetilde{\SL_2\RR}$, do not embed in any limit of hyperbolic geometry in this sense.
\end{abstract}

\section{Introduction}
Following Felix Klein's Erlangen Program, a geometry is given by a pair $(Y,H)$ of a Lie group $H$ acting transitively by diffeomorphisms on a manifold $Y$. Given a manifold of the same dimension as $Y$, a geometric structure modeled on $(Y,H)$ is a system of local coordinates in $Y$ with transition maps in~$H$. The study of deformation spaces of geometric structures on manifolds is a very rich mathematical subject, with a long history going back to Klein and Ehresmann, and more recently Thurston. In this article we are concerned with \emph{geometric transition}, an idea that was recently promoted by Kerckhoff, and studied by Danciger in his thesis \cite{danciger1, danciger2}. 
A geometric transition is a continuous path of geometric structures for which the model geometry $(Y,H)$ abruptly changes to a different geometry $(Y',H')$.  
The process involves the limiting of the two different geometric structures to a common transitional geometry which, in some sense, interpolates the geometric features of the two geometries.
For this to make sense, one must define a notion of geometric limit at the level of homogeneous spaces which describes the basic process by which one homogeneous geometry may transform or limit to another. 
In this article we develop a general framework to describe such geometric transitions, focusing on the special situation 
in which both geometries involved are sub-geometries of a larger ambient geometry $(X,G)$. 
Working in this framework, we then study transitions between certain sub-geometries of real projective geometry, giving an explicit classification in some cases.

The best-known examples of geometric transition arise in the context of Thurston's geometrization program. For example, the transition between hyperbolic and spherical geometry, passing through Euclidean geometry, was studied by Hodgson \cite{Hodgson} and Porti \cite{Porti-98} and plays an important role in the proof of the orbifold theorem by Cooper, Hodgson, Kerckhoff \cite{Cooper_Hodgson_Kerckhoff}, and  Boileau, Leeb, Porti~\cite{Boileau_Leeb_Porti_orbi}. More recently, a transition going from hyperbolic geometry to its Lorentzian analogue, anti de Sitter (AdS) geometry, was introduced by Danciger \cite{danciger1} and studied in the context of cone-manifold structures on Seifert-fibered three-manifolds. The hyperbolic-AdS transition plays an important role in very recent work by Danciger, Maloni, and Schlenker~\cite{Danciger_Maloni_Schlenker} on the classical subject of combinatorics of polyhedra in three-space, which characterizes the combinatorics of polyhedra inscribed in the one-sheeted hyperboloid, generalizing Rivin's famous characterization of polyhedra inscribed in the sphere.
The transition between constant curvature Lorentzian geometries very recently found applications in the setting of affine geometry.  One of the most striking features differentiating affine geometry from Euclidean geometry is the existence of properly discontinuous actions by non-abelian free groups. In three-dimensions, such proper actions of free groups preserve a flat Lorentzian metric and their quotients are called Margulis space-times. In~\cite{Danciger_Gueritaud_Kassel_topology}, Danciger, Gu\'eritaud, and Kassel study Margulis space-times as limits of collapsing complete AdS three-manifolds, giving related characterizations of the geometry and topology of both types of geometric structures. In particular, they give a proof of the tameness conjecture for Margulis space-times \cite{Danciger_Gueritaud_Kassel_topology} (also proved using a different approach by Choi and Goldman \cite{Choi_Goldman}). Each of the results mentioned here involves the construction of a geometric transition in some specialized geometric setting. This paper seeks to broaden the scope of transitional geometry by initiating a classification program for geometric transitions within a very general framework. We expect our results to be useful in a wide array of problems, for example the study of boundaries or compactifications of deformation spaces of geometric structures, and the construction of interesting proper affine actions in higher dimensions.

Let us now illustrate the general context in which the construction of a geometric transition is desirable. 
Consider a sequence $\mathscr Y_n$ of $(Y,H)$ structures on a manifold~$M$, and suppose that as $n \to \infty$, the structures $\mathscr Y_n$ fail to converge, meaning that the charts fail to converge as local diffeomorphisms, even after adjusting by diffeomorphisms of $M$ and coordinate changes in $H$. Of central interest here is the case that the $\mathscr Y_n$ \emph{collapse}: The charts converge to local submersions onto a lower dimensional sub-manifold of $Y$ and the transition maps converge into the subgroup of $H$ that preserves this sub-manifold. 
 Next suppose that $(Y,H)$ is a sub-geometry of $(X,G)$; this means that $Y$ is an open sub-manifold of $X$ and $H$ is a closed subgroup of $G$. 
The sequence $\mathscr Y_n$ of collapsing $(Y,H)$ structures need not collapse as $(X,G)$ structures, because the larger group $G$ of coordinate changes could be used to prevent collapse. In certain cases, one may find a sequence $(c_n) \subset G$, so that the conjugate structures $c_n \mathscr Y_n$ converge to a (non-collapsed) $(X,G)$ structure~$\mathscr Y_\infty$. This limiting structure $\mathscr Y_\infty$ is modeled on a new sub-geometry $(Z,L)$ which is, in a sense to be defined presently, a geometric limit of $(Y,H)$.

Consider two sub-geometries $(Y,H)$ and $(Z,L)$ of $(X,G)$. First, at the level of structure groups, we say $L$ is a \emph{limit} of $H$, if there exists a sequence $(c_n)$ in $G$ so that the conjugates $c_n H c_n^{-1}$ converge to $L$ in the \emph{Chabauty topology} \cite{chabauty} on closed subgroups (i.e. $c_n H c_n^{-1}$ converges to $L$ in the Hausdorff topology in every compact neighborhood of $G$).
If in addition there exists $z \in Z \subset X$ so that $z \in c_n Y$ for all $n$ sufficiently large, then we say $(X,G)$ is a \emph{geometric limit} of $(Y,H)$ as sub-geometries of $(X,G)$. See Section~\ref{sec:limits-of-geometries}.
The description of limit groups and limit geometries in general is a difficult problem. We give a complete classification in certain special cases. 

\subsection*{Symmetric subgroups}
Let $G$ be a semi-simple Lie group of noncompact type with finite center. 
A subgroup $H \subset G$ is called \emph{symmetric} if $H = G^\sigma$ is the fixed point set of an involution $\sigma: G \to G$, or more generally $G^\sigma_0 \subset H \subset G^\sigma$, where $G^\sigma_0$ denotes the identity component of $G^\sigma$.
The coset space $G/H$ is called an affine symmetric space. Affine symmetric spaces have a rich structure theory, generalizing the structure theory of Riemannian symmetric spaces.  In particular there is a Cartan involution $\theta: G \to G$ which commutes with $\sigma$. Let $K = G^{\theta}$ and $\frakg = \frakk \oplus \frakp$ the corresponding Cartan decomposition. The involution $\sigma$ analogously defines a decomposition of $\frakg = \frakh \oplus \frakq$ into $(\pm 1)$--eigenspaces. 
An important tool in determining the limits $L$ of symmetric subgroups $H$ is the following well-known factorization result: Let $\frakb$
be a maximal abelian subalgebra of $\frakp\cap \frakq$. Then any $g \in G$ can be written as $g = kbh$ with 
$k \in K$, $b\in B = \operatorname{Exp}(\frakb)$, and $h \in H$. Moreover $b$ is unique up to conjugation by the Weyl group $W_{H \cap K} := N_{H\cap K} ({\frakb})/ Z_{H\cap K}( {\frakb})$. 
Using this factorization theorem we can characterize all limits of symmetric subgroups as follows (see Section~\ref{sec:symmetric-subgroups} for a more precise version).
\begin{theorem}\label{thm:symmetric-limits}
Let $H$ be a symmetric subgroup of a semi-simple Lie group $G$ with finite center. 
Then any limit $L'$ of $H$ in $G$ is the limit under conjugacy by a one parameter subgroup. More precisely, there exists an $X \in \mathfrak{b}$ such that the limit $L = \lim_{t \to \infty} \exp(tX) H \exp(-tX)$ is conjugate to~$L'$. Furthermore, $$L = Z_H(X) \ltimes N_+(X),$$ where $Z_H(X)$ is the centralizer in $H$ of $X$, and $N_+(X)$ is the connected nilpotent subgroup
$$ N_+(X) := \{g \in G : \lim_{t\to \infty} \exp(tX)^{-1} g \exp(tX) = 1 \}.$$
\end{theorem}

In the special case when  $H = K$ (in other words $\sigma$ is a Cartan involution), the limit groups are determined by Guivarc'h--Ji--Taylor \cite{GJT} and also by Haettel \cite{Haettel2}.  Moreover in these articles the Chabauty-compactification of $G/K$ is shown to be isomorphic to the maximal Satake-Furstenberg compactification. 
The analysis leading to Theorem~\ref{thm:symmetric-limits} bears a lot of similarity with the analysis in  \cite{Gorodnik_Oh_Shah}, where the maximal Satake-Furstenberg compactification for affine symmetric spaces $G/H$ is defined. We would like to raise the question whether the Chabauty-compactification of $G/H$ is homeomorphic to its maximal Satake-Furstenberg compactification.

The groups $G = \PGL_n \KK$  with $\KK = \RR$ or $\CC$ are of particular interest, since they are the structure group for projective geometry. In this case, 
Theorem~\ref{thm:symmetric-limits} implies that the limits of symmetric subgroups have a nice block matrix form. Let $H \subset \PGL_{n}\KK$ be a symmetric subgroup and let $L$ be a limit of $H$.
 Then, there is a decomposition $\KK^n = E_0 \oplus \cdots \oplus E_k$ of $\KK^n$ with respect to which $L$ has the following block form:
 $$\left(\begin{array}{ccccc}
A_1 & 0 & 0 &\cdots & 0\\
* & A_2 & 0 &\cdots & 0\\
* & * & A_3&\cdots& 0\\
\cdots &\cdots &\cdots &\cdots &\cdots \\
* & * & *  &\cdots & A_k
\end{array}\right).$$
Here, the blocks denoted $*$ are arbitrary, and the diagonal part $\operatorname{diag}(A_1,\ldots,A_k)$ is an element of~$H$. 
The groups  $\PO_n\CC$, $\PP(\GL_p\CC \times \GL_q\CC)$, $\PP\Sp(2m,\CC)$, where $n= 2m$,  $\PGL_n\RR$, $\PU(p,q)$, where  $n = p+q$, and $\SL(m,\mathbb{H})$, where $n= 2m$  are some of the symmetric subgroups of $\PGL_n\CC$. 
The symmetric subgroups of $\PGL_{n} \RR$ are $\PP(\GL_p\RR \times \GL_q\RR)$, $\PO(p,q)$, where $p+q = n$,  or $\PP\Sp(2m, \RR)$ and $\PP(\GL(m,{\mathbb C}))$, where $n = 2m$.  See Section~\ref{sec:symmetric-subgroups} for a full characterization of the limit groups in each of these cases.

\subsection*{Limits of constant curvature semi-Riemannian geometries.}
Let $\beta$ denote a quadratic form on $\RR^n$ of signature $(p,q)$, meaning 
$\beta\sim-I_p\oplus I_q$. 
The group $\PP\Isom(\beta) = \PO(p,q)$ acts transitively on the domain
$$\XX(p,q) = \{ [x] \in \RP^{n-1} : \beta(x) < 0 \} \subset \RP^{n-1},$$
with point stabilizer isomorphic to $\OO(p-1,q)$. The geometry $(\XX(p,q), \PO(p,q))$ is the projective model for semi-Riemannian geometry of constant curvature, of dimension $p+q-1$ and of signature $(p-1,q)$. 
In the cases $(p,q)$ is $(n,0), (1,n-1), (n-1,1)$ or $(2,n-2)$ we obtain spherical geometry, hyperbolic geometry, de Sitter geometry, and anti de Sitter geometry respectively.

By applying Theorem~\ref{thm:symmetric-limits}, we characterize the limits of these constant curvature semi-Riemannian geometries inside real projective geometry. Here is a brief description of the limit geometries.
A~{\em partial flag} ${\flag}=\{V_0,V_1,\cdots V_{k+1}\}$ of $\RR^n$
is a descending chain of vector subspaces
$$\RR^n=V_0\supset V_1\supset\cdots \supset V_k\supset V_{k+1}=\{0\}.$$ 
A {\em partial flag of quadratic forms} $\pfbeta=(\beta_0,\cdots,\beta_k)$ on ${\flag}$
 is a collection of non-degenerate quadratic forms $\beta_i$ defined on each quotient
$V_i/V_{i+1}$ of the partial flag. 
Define $\Isom(\pfbeta,{\flag})$ to be the group of linear transformations which preserve $\flag$ and induce an isometry of each $\beta_i$, and denote its image in $\PGL_n \RR$ by $\PP\Isom(\pfbeta, \flag)$. 
 Define the domain $\XX(\pfbeta) \subset \RP^{n-1}$ by $$\XX(\pfbeta) := \{ [x] \in\RP^{n-1} : \beta_0(x) < 0\}.$$
Then $\PP \Isom(\flag, \pfbeta)$ acts transitively on $\XX(\pfbeta)$. 
When the flag and quadratic forms are adapted to the standard basis, we denote $\XX(\pfbeta)$ and $\Isom(\pfbeta)$ by 
\begin{align*}\label{PFG-notation}
\XX(\pfbeta) &=: \XX((p_0,q_0),\ldots,(p_k,q_k)),\\
\Isom(\pfbeta) &=: \OO((p_0,q_0),\ldots,(p_k,q_k))\\
&= \left(\begin{array}{ccccc}
\OO(p_0,q_0) & 0 & 0 &\cdots & 0\\
* & \OO(p_1,q_1) & 0 &\cdots & 0\\
* & * & \OO(p_2,q_2) &\cdots& 0\\
\cdots &\cdots &\cdots &\cdots &\cdots \\
* & * & * &\cdots & \OO(p_k,q_k)
\end{array}\right),
\end{align*}
where $*$ denotes an arbitrary block. Note that $\XX(\pfbeta)$ is non-empty if and only if $p_0 > 0$. As a set, the space $\XX((p_0,q_0)\ldots (p_k,q_k))$ depends only on the first signature $(p_0,q_0)$ and the dimension $n = \sum_i (p_i + q_i)$. However, we include all $k$ signatures in the notation as a reminder of the structure determined by $\PO((p_0,q_0),\ldots,(p_k,q_k))$.

\begin{theorem}\label{thm:limits-Hpq_intro}
The limits of the constant curvature semi-Riemannian geometries $(\XX(p,q), \PO(p,q))$ inside $(\RP^{n-1}, \PGL_n \RR)$ are all of the form $(\XX(\pfbeta), \PP \Isom(\flag, \pfbeta))$. Further, $\XX(\pfbeta)$ is a limit of $\XX(p,q)$ if and only if $p_0 \neq 0$, and the signatures $((p_0,q_0),\ldots,(p_k,q_k))$ of $\pfbeta$ partition the signature $(p,q)$ in the sense that
\begin{equation*}
p_0 + \cdots + p_k = p \ \ \text{  and  } \ \ q_0 + \cdots + q_k = q,
\end{equation*}
after exchanging $(p_i, q_i)$ with $(q_i, p_i)$ for some collection of indices~$i$ in $\{1,\ldots,k\}$ (the first signature $(p_0,q_0)$ must \emph{not} be reversed).
\end{theorem}

\noindent See Section~\ref{sec:Hpq} for a detailed discussion of these partial flag geometries.

\subsection*{The Thurston geometries} 

One motivation for our work is the study of transitions between the eight three-dimensional \emph{Thurston geometries}, homogeneous Riemannian geometries which play an essential role in the classification of compact three-manifolds. Since each of the eight geometries (almost) admits a representation in real projective geometry \cite{molnar}, \cite{thiel}, it is natural to study transitions between them in the projective setting. 

From the point of view of Thurston's picture of hyperbolic Dehn surgery space, one expects to find five of the eight Thurston geometries as limits of hyperbolic geometry. 
There are various methods, due to Porti and collaborators \cite{Porti-98,Porti-02,HPS}, to realize Euclidean, Nil, and Sol geometry structures as limits of certain families of collapsing hyperbolic structures. However, efforts to realize the Thurston geometries which fiber over the hyperbolic plane, namely $\HH^2 \times \RR$ and $\widetilde{\SL_2 \RR}$, as limits of three-dimensional hyperbolic geometry have so far proved fruitless.
Many Seifert-fibered three-manifolds which admit a structure modeled on either $\HH^2\times \RR$ or $\widetilde{\SL_2 \RR}$ also admit hyperbolic cone-manifold structures with cone angles arbitrarily close to $2\pi$. Commonly in examples, such structures are found to collapse down to a hyperbolic surface (the base of the Seifert fibration) as the cone angle increases to $2\pi$. 
Recent work of Danciger \cite{danciger1, danciger2} shows that in this context, the most natural sequence of conjugacies in projective space yields a non-metric geometry called \emph{half-pipe} geometry as limit. However, Danciger's construction does not rule out the possibility that some other clever sequence of conjugacies could produce $\HH^2 \times \RR$ or $\widetilde{\SL_2 \RR}$ geometry as limit. As an application of Theorem~\ref{thm:limits-Hpq_intro}, we enumerate the limits of hyperbolic geometry inside of projective geometry and prove:

\begin{theorem}
\label{cor:Thurston-geoms_intro}
 The Thurston geometries which locally embed in limits of hyperbolic geometry (within real projective geometry) are 
: ${\mathbb E}^3$, Solv geometry, and Nil geometry. In particular, neither $\HH^2 \times \RR$ nor $\widetilde{\SL_2 \RR}$ locally embed into any limit of hyperbolic geometry.
\end{theorem}

\noindent In future work we intend to give a complete description of the possible transitions between the eight Thurston geometries. 

\subsection*{Structure of the paper}
In Section~\ref{sec:limits-of-geometries} we introduce the notion of (geometric) limits of groups and limits of geometries and revisit the transition from hyperbolic geometry through Euclidean geometry  to spherical geometry (within real projective geometry). 
In Section~\ref{sec:limit-groups} we describe several notions of limits of groups, illustrate their differences, show when they agree, and describe some basic properties of geometric limits of real algebraic Lie groups. 
In Section~\ref{sec:symmetric-subgroups} we recall the basic structure theory of affine symmetric spaces, and prove (a more descriptive version of) Theorem~\ref{thm:symmetric-limits}, determining the limit groups of symmetric subgroups. We then apply the theorem to give explicit descriptions of the limits of symmetric subgroups of the general linear group. 
The limit geometries of the semi-Riemannian real hyperbolic geometries in terms of partial flags geometries are described in Section~\ref{sec:pfqf}. In the end of that section we discuss the applications to Thurston geometries.

\subsection*{Acknowledgements}
This work  grew out of discussions at an AMS Mathematical Research Community at Snowbird in 2011, and we thank the {\em AMS} and {\em NSF} for their supporting this program.  
Some of this work occured during the trimester ``Geometry and Analysis of Surface Groups'' at {\em IHP},  and we 
gratefully appreciate their support and the stimulating environment there. 
We thank the research network ``Geometric Structures and Representation Varieties" (GEAR) , funded by the NSF under the grant numbers DMS 1107452, 1107263, and 1107367, for supporting reciprocal visits to Princeton, Santa Barbara and Heidelberg. 

We thank { Marc Burger} for interesting discussions, and in particular for suggesting to work in the context of affine symmetric spaces, and pointing us to the relevant structure theorems. 
We also profited from discussions with {Thomas Haettel}.
We thank Benedictus Margeaux for providing a proof of Lemma~\ref{lem:centralizers} on math overflow. 

 D.C. was partially supported by NSF grants DMS-0706887,  1207068 and  1045292.
J.D. was partially supported by the National Science Foundation under the grant DMS 1103939.
A.~W. was partially supported by the National Science Foundation under agreement No. DMS-1065919 and 0846408, by the Sloan Foundation, by the Deutsche Forschungsgemeinschaft, and by the ERCEA under ERC-Consolidator grant no. 614733.

\section{Limits of Geometries}
\label{sec:limits-of-geometries}

\begin{definition}
A {\em geometry} is a pair $(X,G)$ where $G$ is a Lie group acting transitively by analytic maps on
a connected, smooth manifold $X$.
\end{definition}
\noindent The requirement in the definition that the action be transitive implies that $X$ identifies with $G/G_x$, where
$G_x$ denotes the stabilizer of a point $x \in X$. Note that we do not require the point stabilizer $G_x$ to be compact.
Some examples of geometries are 
\begin{enumerate}
\item[(1)] Euclidean geometry $\mathbb E^n$: The space $X  = \RR^n$ and the structure group $G$ is the semi-direct product of the orthogonal group $\OO(n)$ and the translation group $\RR^n$.
\item[(2)] Spherical geometry $\mathbb S^n$: The space $X$ is the unit sphere in $\RR^{n+1}$ and the group $G$ is the orthogonal group $\OO(n+1)$.
\item[(3)] Hyperbolic geometry $\mathbb H^n$: The space $X  = \{ [\mathbf{x}] \in \RP^{n}: -x_{n+1}^2 + x_1^2+ \cdots + x_n^2 < 0\}$ is the set of negative lines with respect to the standard quadratic form of signature $(1,n)$. The structure group $G = \PO(1,n)$ is the group of projective transformations preserving $X$.
\item[(4)] Real projective geometry: $X =  \RP^n$, $G = \PGL_{n+1} \RR$.
\end{enumerate}

\begin{definition}
Given geometries $(X,G)$ and $(X', G')$ a \emph{morphism} $(X,G) \to (X',G')$ is a Lie group homomorphism $\Phi: G \to G'$ such that for some (and hence any) $x \in X$ there is an $x' \in X'$ such that $\Phi(G_x) \subset G_{x'}'$. The map $\Phi$ induces an analytic map $F: X \to X'$ defined by $F(x) = x'$ and the property that $F$ is $\Phi$-equivariant: $$F(g\cdot y) = \Phi(g) F(y).$$ The map $\Phi$ defines an \emph{isomorphism} of geometries if $\Phi$ is an isomorphism of Lie groups and $\Phi(G_x) = G_{x'}'$. If $\Phi$ is surjective, we say $(X,G)$ \emph{fibers} over $(X', G')$.
\end{definition}

Recall that a local homomorphism $\varphi: G \dashrightarrow G'$ of Lie groups is a map $\varphi: V \to G$, defined on a neighborhood $V$ of the identity in $G$, such that $\varphi(gh) = \varphi(g)\varphi(h)$ and $\varphi(g)^{-1} = \varphi(g^{-1})$  whenever all terms are defined. A local homomorphism $\varphi$ is a local isomorphism if $\varphi$ is locally injective, meaning injective when restricted to some small neighborhood of the identity in $G$, and locally surjective, meaning has image containing a small neighborhood of the identity in $G'$. Note that if $\frakg$ and $\frakg'$ denote the Lie algebras of $G$ and $G'$, then the differential $\varphi_*: \frakg \to \frakg'$ of a local homomorphism of Lie groups $\varphi$ is a homomorphism of Lie algebras $\varphi_*: \frakg \to \frakg'$ and conversely any homomorphism of Lie algebras is the differential of a local homomorphism of Lie groups as above.  
\begin{definition}
A \emph{local morphism of geometries} $(X,G) \dashrightarrow (X', G')$ is a local homomorphism $\varphi : G \dashrightarrow G'$ such that for some (and hence any) $x \in X$, there is an $x' \in X'$ with the property that the restriction of $\varphi$ to $G_x$ has image in $G_{x'}'$. The local homomorphism $\varphi$ induces a local analytic map $f: X \dashrightarrow X'$, defined on a neighborhood of $x$, which is locally $\varphi$-equivariant, meaning $f(g \cdot y) = \varphi(g)f(y)$ whenever all terms are defined. 
The local morphism is a \emph{local isomorphism} if $\varphi$ is a  local isomorphism $G \dashrightarrow G'$ and $\varphi$ restricts to a local isomorphism $G_x \dashrightarrow G_{x'}'$.
\end{definition}

Note that given $\varphi$ as in the definition, the differential $\varphi_*: \frakg \to \frakg'$ satisfies that $\varphi_*(\frakg_x) \subset \frakg_{x'}'$, where $\frakg_x$ and $\frakg_{x'}'$ are the Lie algebras of the point stabilizers $G_x$ and $G_{x'}'$. Conversely a Lie algebra homomorphism $\frakg \to \frakg'$ taking infinitesimal point stabilizers into infinitesimal point stabilizers (which could be called an infinitesimal morphism of geometries) determines a local morphism of geometries.  The map $\varphi: G \dashrightarrow G'$ determines a local isomorphism if and only if $\varphi_*$ is an isomorphism and $\varphi_*(\frakg_x) = \frakg_{x'}'$.

If $G$ is connected then the {\em universal cover} $(\tilde X,\tilde G)\longrightarrow(X,G)$ of a geometry $(X,G)$ is defined as follows: Let $G_x \subset G$ be the stabilizer of a point $x \in X$ so that $X = G/G_x$. Then $\tilde G \to G$ is the universal covering Lie group of $G$, and $\tilde X = \tilde G /\tilde G_x$, where $\tilde G_x \subset \tilde G$ is the identity component of the inverse image of $G_x$. Indeed $\tilde X \to X$ is the universal cover of $X$. Note that the action of $\tilde G$ on $\tilde X$ might not be faithful, even if the action of $G$ on $X$ was faithful. In this case, one may replace $\tilde G$ with its quotient by the kernel of the action. Every geometry is locally isomorphic to its universal cover. A local morphism  (resp. local isomorphism) of geometries induces a morphism (resp. isomorphism) of the universal covering geometries.


\begin{definition}
The geometry $(Y,H)$ is a {\em subgeometry} of  $(X,G)$, written $(Y,H)\subset(X,G)$, if $H$
is a closed subgroup of $G$ 
and $Y$ is an open subset of $X$ on which $H$ acts transitively. We say that a geometry $(Y,H)$ \emph{locally embeds} in $(X,G)$ if $(Y,H)$ is locally isomorphic to a subgeometry $(Y',H') \subset (X,G)$.
\end{definition}
For example, both hyperbolic and Euclidean geometry, in the forms described above, are subgeometries of real projective geometry.
Spherical geometry is a two fold covering of a subgeometry of projective geometry, and therefore spherical geometry locally embeds in projective geometry but it is not a subgeometry.
Similarly, the Thurston geometry known as $\widetilde{\SL_2{\mathbb R}}$ is not a subgeometry of projective geometry, but it does locally embed.

In this article we are concerned with limits of geometries, in particular with limits of subgeometries of a given geometry $(X,G)$. First we introduce the notion of limit of closed subgroups. 

\begin{definition}
\label{def:limit-subgroups}

\noindent
\begin{enumerate}
\item A sequence $H_n$ of closed subgroups of a Lie group $G$ \emph{converges geometrically} to a closed subgroup $L$ if every $g \in L$ is the limit of some sequence $h_n \in H_n$, and if every accumulation point of every sequence $h_n \in H_n$ lies in $L$. We also say that $L$ is the \emph{geometric limit} of the sequence $H_n$.
Note that $L$ is the geometric limit of $H_n$ if and only if $H_n$ converges to $L$ in the Chabauty topology on closed subgroups.
\item We say $L$ is a \emph{conjugacy limit} (or just \emph{limit}) of $H$ if there exists a sequence $c_n \in G$ so that the conjugate groups $H_n = c_n H c_n^{-1}$ converge geometrically to $L$.
\end{enumerate}
\end{definition}

\noindent The set ${\mathcal C} (G)$ of closed subgroups with the Chabauty topology is a compact space \cite{chabauty}, \cite{bridson}, \cite{delaharpe}, so for every sequence of closed subgroups there is some subsequence which has a geometric limit.
We on the set of all closed subsets of a non-compact space is commonly defined to be the topology of Hausdorff convergence in compact neighborhoods. The Chabauty topology is simply the subspace topology on the set of closed subgroups.

\begin{definition}
\label{def:limit-subgeom}
\noindent
\begin{enumerate}
\item
A sequence of subgeometries $(Y_n,H_n)\subset (X,G)$ {\em converges} to the subgeometry
$(Z, L)\subset (X,G)$ if $H_n$ converges geometrically to $L$ and 
there exists $z\in Z \subset X$ such that for all $n$ sufficiently large $z\in Y_n$.
\item We say that a subgeometry $(Z,L)$ is a \emph{conjugacy limit} (or just \emph{limit}) of $(Y,H)$ if there exists a sequence $g_n \in G$ so that the sequence of conjugate subgeometries $(g_n Y, g_n H g_n^{-1})$ converges to $(Z,L)$.
\end{enumerate}
\end{definition}

The motivating situation, as described in the introduction, is that of collapsing $(Y,H)$ structures on a manifold: 
structures for which each chart (or alternatively the developing map) collapses to a local submersion onto a lower-dimensional subset $N$ of $Y$ and each transition map (alternatively the holonomy representations) converges into some smaller subgroup $P \subset H$ that preserves $N$.
The goal is to conjugate the $(Y,H)$-structures inside of $(X,G)$, so that the charts (developing maps) no longer collapse and the transition maps (holonomy representations) converge into some limit group $L$ of $H$ which contains~$P$. Then, setting $Z = L \cdot N$, the geometry $(Z, L)$ is a limit of $(Y,H)$ in the sense of Definition~\ref{def:limit-subgeom} and the limiting geometric structure is a $(Z,L)$ structure.

\subsection{The transition from spherical to Euclidean to hyperbolic}
\label{sec:spherical-Euclidean} 
Let us now illustrate the definitions in a familiar example.

Consider the path of quadratic forms $\beta_t= -x_{n+1}^2- t(x_1^2+\cdots+x_n^2)$ on ${\mathbb R}^{n+1}$, and assume $t \geq 0$.
These quadratic forms define sub-geometries $(\XX(\beta_t), \PO(\beta_t))$ of projective geometry where 
$$\XX(\beta_t) = \{ [x] \in \RP^n : \beta_t(x) < 0\},$$
$$\PO(\beta_t) = \operatorname{P}\left\{ A \in \GL(n+1) : A^* \beta_t = \beta_t \right\}.$$
For all $t > 0$, $\PO(\beta_t)$ is conjugate to $\PO(\beta_1)$ which is the standard copy of $\PO(n+1)$, and the geometry $(\XX(\beta_t), \PO(\beta_t))$ is conjugate to the standard realization (found at $t=1$) of spherical geometry $\mathbb S^n$ as a (covering) sub-geometry of projective geometry. The element $c_t \in \PGL(n+1)$ conjugating $(\XX(\beta_1), \PO(\beta_1))$ to $(\XX(\beta_t), \PO(\beta_t))$ is 
the diagonal matrix $c_t = \operatorname{diag}(1/\sqrt{t}, \ldots,1/\sqrt{t}, 1)$ with the first $n$ diagonal entries equal to $1/\sqrt{t}$ and the final diagonal entry equal to one (note $c_t^*\beta_t = \beta_1$). This corresponds to scaling the plane $\RR^n \subset \RR^{n+1}$ spanned by the first $n$ coordinate directions.

At time $t = 0$, the quadratic form becomes degenerate. The group preserving $\beta_0$ is simply the group of matrices that preserve the last coordinate $|x_{n+1}|$; this is a copy of the affine group
$$\operatorname{Aff}(n)=\left\{\left(\begin{matrix}
A & b\\
0 & 1 
\end{matrix}\right)\right\}\subset \PGL(n+1,{\mathbb R}).$$
However, the affine group is \emph{not} the limit of the groups $\PO(\beta_t)$ as $t \to 0^+$.
In fact, the limit  of the conjugate subgroups $\PO(\beta_t))$ as $t \to 0^+$ is the group of Euclidean isometries. In order to simplify the discussion, let us demonstrate this at the level of Lie algebras.
The Lie algebra $\mathfrak{so}(\beta_t)$ of $\PO(\beta_t)$ is conjugate to the Lie algebra $\mathfrak{so}(\beta_1) = \mathfrak{so}(n+1)$:

\begin{align*}
\mathfrak{so}(\beta_t) &= 
\left(
\begin{matrix}
\sqrt{t}^{-1}&    &  & \\
  & \sqrt{t}^{-1} &  &  \\
 &  & \ddots & \\
  &          & & 1
\end{matrix}
\right)
\left(\begin{matrix} 
0 &    &   a_{ij} & \vdots\\
 &  \ddots &  &  b\\
-a_{ji}&  & \ddots & \vdots\\
\hdots &    -b^T    &\ldots & 0
\end{matrix}\right)
   \left(
\begin{matrix}
\sqrt{t} &    &  & \\
  &  \sqrt{t} &  &  \\
 &  & \ddots & \\
  &          & & 1
\end{matrix}
\right)\\
&=
\left(\begin{matrix} 
0 &    &   a_{ij} & \vdots\\
 &  \ddots &  &  \sqrt{t}^{-1}b\\
-a_{ji}&  & \ddots & \vdots\\
\hdots &    \sqrt{t} b^T      &\ldots & 0
\end{matrix}\right).
\end{align*}
It is easy to read off the limit Lie algebra via the following heuristic reasoning. To find an element of the limit Lie algebra, we are allowed to vary the entries $a_{ij}, b$ of the matrix as $t \to 0^+$ in any way that produces a limit. Since $\sqrt{t}^{-1} \to \infty$, it follows that we must have $b = O(\sqrt{t})$. Thus the limit Lie algebra has the form:

$$
\lim_{t \to0^+} \mathfrak{so}(\beta_t) = 
\left(\begin{matrix} 
0 &    &   a_{ij} & \vdots\\
 &  \ddots &  &  b'\\
-a_{ji}&  & \ddots & \vdots\\
\hdots &    0      &\ldots & 0
\end{matrix}\right) = \mathfrak{isom}(\EE^n)
$$
where $b'$ can be any column vector. We recognize this limit as the Lie algebra of the subgroup of the affine group preserving the standard Euclidean metric on the affine patch $x_{n+1} \neq 0$.
In fact, it is only slightly more difficult to show that the limit of the Lie groups $\PO(\beta_t)$ is indeed the group of Euclidean isometries 
$$\Isom(\EE^n) = \operatorname{P}\left( \begin{pmatrix} \OO(n) & \\ & \pm1 \end{pmatrix} \ltimes \begin{pmatrix} I_n &\RR^n \\ &1 \end{pmatrix}\right).$$
To determine the limit of the homogeneous spaces $\XX(\beta_t)$, we use Definition~\ref{def:limit-subgeom}. Consider any point $z$ in the affine patch $\EE^n = \{ [x] : x_{n+1} \neq 0\}$. Then, of course $z$ is in $\XX(\beta_t) = \RP^n$ for all $t > 0$. Note that if we choose $z$ to be the usual origin of $\EE^n$, then $z$ is fixed by $c_t$. 
 The notion that the limit of a constant sequence of spaces ($\RP^n$) could be anything other ($\EE^n$) than that space again may seem counter-intuitive. However, the important thing to realize is that the orbit of $z$ under the groups $\PO(\beta_t)$ is $\RP^n$ while the orbit of $z$ under the limit group $\Isom(\EE^n)$ is now the smaller space $\EE^n$. This is indeed the relevant notion of limit in the context of geometric structures.
 
Next, for $t < 0$, the $\beta_t$ have signature $(1,n)$ and the corresponding sub-geometries $(\XX(\beta_t), \PO(\beta_t))$ are all conjugate to the standard copy $(\XX(\beta_{-1}), \PO(\beta_{-1}))$ of the projective model for hyperbolic space $\HH^n$. The reasoning above applies similarly in this case to show that the limit of $\PO(\beta_t)$ as $t \to 0^-$ is again the group $\Isom(\EE^n)$ of Euclidean isometries. In this case the spaces $\XX(\beta_t)$ are expanding balls in $\RP^n$ which eventually engulf any point $z$ in the affine patch $\EE^n = \{ [x] : x_{n+1} \neq 0\}$. Thus $(\EE^n, \Isom(\EE^n)$ is the limit of $(\XX(\beta_t), \PO(\beta_t))$ as $t \to 0^-$.

It is worth noting that this transition of homogeneous spaces can be seen nicely in terms of certain quadric hyper-surfaces in $\RR^{n+1}$. For, the level sets of the quadratic forms $\beta_t$ are either ellipsoids if $t > 0$, or hyperboloids of two sheets if $t < 0$. Any such level set $\beta_t = -k$ is preserved by the lift $\OO(\beta_t)$ of $\PO(\beta_t)$ and projects (two to one) to $\XX(\beta_t)$ in projective space; hence it gives a nice model for the same geometry. In the case $t=1$, the level set $\beta_t = -1$ is the unit sphere and describes the standard model for spherical geometry, while in the case $t= -1$, the level set $\beta_t = -1$ is the standard hyperboloid model for hyperbolic geometry. Note that for all $t$, the level set $\beta_t = -1$ contains the two points $(0,\ldots,0,\pm1)$. As $t \to 0$ (from either direction), the limit of the surfaces $\beta_t = -1$, in the topology of Hausdorff convergence on compact sets, is the surface $\beta_0 = -1$, which is two parallel affine hyper-planes $x_{n+1} = \pm 1$. See Figure~\ref{fig:ellipsoids}. If one wishes, one may define invariant metrics on the $\XX(\beta_t)$ which transition from uniform positive curvature to uniform negative curvature as $t$ changes from positive to negative. However, its important to note that there is no natural such continuous path of metrics defined from the ambient geometry. In particular, the natural metric on the surfaces $\beta_t = -1$ induced by the quadratic forms $\beta_t$ have curvature $+1$ when $t > 0$ and curvature $-1$ when $t < 0$ (and of course, $\beta_0$ itself does not define any metric). Hence the projective geometry formulation of the transition from spherical to Euclidean to hyperbolic is independent of any metric formulation.

\begin{figure}[h]
{\centering

\def\svgwidth{3.3in}
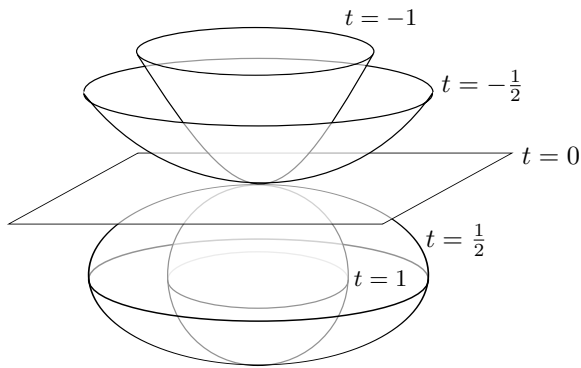

}
\caption{Ellipsoids and hyperboloids defined by $\beta_t = -1$ limit as $t \to 0$ to a pair of opposite affine hyperplanes (only one of these planes is drawn). \label{fig:ellipsoids}}
 \end{figure}

\section{Limits of Groups}\label{sec:limit-groups}
Classifying the limits of closed subgroups $H$ of $G$ is the central problem when classifying the limits of a sub-geometry $(Y,H)$ in $(X,G)$.
Thus in the next two sections, we restrict our attention to limits of Lie groups and momentarily forget about homogeneous spaces. 

We are mainly interested in \emph{geometric limits} of Lie subgroups (Definition~\ref{def:limit-subgroups}). 
However, there are other inequivalent definitions of limit of a group, and it is helpful to understand
how they differ and when they coincide. We explore a few of these alternative notions in Section~\ref{sec:other-notions} and illustrate them through examples in Section~\ref{sec:examples}. In Section~\ref{sec:limit-props} we derive some basic properties of geometric limits of linear algebraic Lie groups.

\subsection{Various notions of limit}
\label{sec:other-notions}
Let $H_n$ be a sequence of closed Lie subgroups, of constant dimension, of the Lie group $G$. 
We introduced the {\em geometric limit} of $H_n$ in Definition~\ref{def:limit-subgroups}. Here are several related notions of limit.

\begin{enumerate}
\item[(1)] The {\em connected  geometric  limit} $\lim_0H_n$ is the connected component of the identity in the geometric limit.
In general this is different than the geometric limit of the connected component of the identity.

\item[(2)] In the specific case that $H_n = c_n H c_n^{-1}$ are all conjugate, we may define the {\em local geometric  limit}, denoted $\operatorname{local}$-$\lim H_n$, as the union of the 
geometric limits of conjugates $c_n C c_n^{-1}$ of compact neighborhoods $C \subset H$ of the identity. It might have smaller dimension than $H$.
This limit is contained in the geometric limit, but excludes conjugates of elements moving to infinity.
One may also define a notion of local geometric limit with respect to a subgroup $P \subset H$; this means the union of geometric limits of neighborhoods of the form $P \cdot C$ for $C$ a compact neighborhood of the identity.

\item[(3)] Very much related to the local geometric limit is the notion of expansive limit.
Again, we work in the case that $H_n = c_n H c_n^{-1}$ are all conjugate and we consider a subgroup $P \subset H$, such that $c_n P c_n^{-1} = P$. 
A (local) geometric limit is called an \emph{expansive limit}, if every element of the limit group $L$ is of the form  $\ell = \lim c_n h_n c_n^{-1}$, for some sequence $h_n \in H$ with $h_n \to h_\infty \in P$. Intuitively, an expansive limit is a limit obtained by blowing up an infinitesimal neighborhood of $P$. Expansive limits are often the relevant limits to study in the context of collapsing geometric structures and geometric transitions. See discussion following Definition~\ref{def:limit-subgeom}. In Section~\ref{sec:symmetric-subgroups}, we will demonstrate that all limits of symmetric subgroups of semi-simple Lie groups are expansive.

\item[(4)] The {\em Lie algebra limit}  is the Lie sub-algebra of $\frakg = \operatorname{Lie}(G)$ obtained from the limit of the sequence of 
Lie subalgebras ${\mathfrak h}_n\subset {\mathfrak g}$ of the
subgroups $H_n$. As the $\mathfrak{h}_n$ are vector sub-spaces of $\frakg$, we may (up to subsequence) extract a limit $\mathfrak{l}$, which is a vector subspace of the same dimension as the $\mathfrak{h}_n$ and in fact a Lie sub-algebra. Then $\mathfrak{l}$ defines a local group near the identity in $G$ and generates a subgroup that we call $\fraka$-$\lim H_n$. It has the same dimension as the $H_n$.
Note, however, that this subgroup might not be closed. We call the closure of $\fraka$-$\lim H_n$ the {\em algebraic limit}, denoted $a$-$\lim H_n$. It is a connected Lie subgroup which might have
larger dimension than the $H_n$.

\item[(5)] Let us mention a related, but distinct notion of limit at the level of Lie algebras.
The Lie algebra $\frakh$ is determined by its Lie bracket 
$[\ ,\ ]:\frakh\times \frakh \to \frakh$. This is
a bilinear map and so is determined by its values on a basis, and thus by
 a finite collection of {\em structure constants}. One may
continuously change these constants, taking care to ensure they still determine a Lie algebra,
and in this way pass between the Lie algebras of different geometries. A path corresponding to
change of bases leads to the theory, due to In\"{o}n\"{u}-Wigner, of \emph{contractions} of Lie algebras which is useful in some physics contexts \cite{burde}, \cite{Wigner}.
This notion is independent of any embedding of $\frakh$ into a larger Lie algebra.

\item[(6)] Although we do not pursue it here, there is also a notion of geometric limit of a Lie group that does
not involve conjugacy inside a larger group. It is based on the idea  that an arbitrarily large compact 
subset of the limit
is almost isomorphic to a (possibly small) subset of the original group.
 A Lie group $L$  is an {\em intrinsic limit} of a Lie group
$G$ if for every compact subset $C\subset L$ and $\epsilon>0$ there is an open set $U \subset G$ and an immersion
$f:U\longrightarrow L$ such that $f(U)\supset C$  
and $f$ is $\epsilon$-close to an isomorphism in the sense that  if $a,b,ab\in C$ there are $\alpha,\beta,\alpha\beta\in U$ with $f(\alpha)=a$, $f(\beta)=b$ and
$d_L(f(\alpha\beta),ab)<\epsilon$. 
Here $d_L$ is a metric on $L$. 
It is easy to see that if $L$ is a geometric limit of a closed subgroup $H$ of a Lie group $G$
then $L$ is an intrinsic limit of $H$ in this sense. Moreover, this definition extends in an obvious way
to pairs $(H,K)$ with $K$ a closed subgroup of $H$ and gives a notion of limit of the homogeneous geometry $H/K$ independent
of an ambient geometry $G/R$.

\end{enumerate}

\subsection{Examples of limits of groups}
\label{sec:examples}
We now give some examples to demonstrate various possible behavior of limits.
In what follows we make use of three $1$-parameter subgroups of $SL(2,{\mathbb R})$:

$$D(t) =\left(\begin{array}{cc} \cosh(t) & \sinh(t)\\ \sinh(t) & \cosh(t) \end{array}\right),
  \qquad\qquad 
 R(t)=\left(\begin{array}{cc} \cos(t) & -\sin(t)\\ \sin(t) & \cos(t)\end{array}\right),
 \qquad\qquad 
 P(t)=\left(\begin{array}{cc} 1 & t\\ 0 & 1\end{array}\right),$$
 All groups in this section will be groups of matrices, described in terms of at most three parameters $t,s,u \in \RR$. Whenever one of these parameters is present, the reader is meant to take the union over such matrices for all possible values of the parameters. We will abuse notation as such throughout, because it is cumbersome to write the definitions properly using set notation.

\begin{enumerate}
\item[(1)] Let  $c_n = \begin{pmatrix} n & 0 \\ 0 & n^{-1} \end{pmatrix}$. 
Then as $n \to \infty$, the sequence of conjugates $c_n D(t) c_n^{-1} \subset \SL(2,{\mathbb R})$ converges (for all notions of limits introduced above) to the parabolic subgroup $P(t)$. 

\item[(2)] In $\SL(2,{\mathbb R})$ the sequence of conjugates $c_n R(t)c_n^{-1}$, with $c_n$ as in (1), converges geometrically to the parabolic subgroup $\pm P(t)$ with two connected components. This illustrates that the {\em geometric limit} of a connected group might not be connected.

\item[(3)] A subgroup of $\SL(3,{\mathbb R})$ containing non-diagonalizable
elements may have a {\em geometric limit} containing only diagonal elements:
$$\lim_{n\to \infty}\left(\begin{array}{ccc}
1/n & 0 & 0\\
0 & n & 0\\
0 & 0 & 1\end{array}\right)
\left(\begin{array}{ccc}
1 & t & 0\\
0 & 1 & 0\\
0 & 0 & e^{-2t}\end{array}\right)
\left(\begin{array}{ccc}
1/n & 0 & 0\\
0 & n & 0\\
0 & 0 & 1\end{array}\right)^{-1}=
\left(\begin{array}{ccc}
1 & 0 & 0\\
0 & 1 & 0\\
0 & 0 & e^{-2t}\end{array}\right)$$

\item[(4)] Next, we give an example of a one-dimensional group with a two-dimensional conjugacy limit. 
  Let $H$ be the 1-parameter closed subgroup of $\SL(4,{\mathbb R})$ defined by
 $$H=\left(\begin{array}{cc}
 P(t) &  0\\
 0 & R(t)
 \end{array}\right)$$
The geometric limit under conjugacy by
$c_n=\operatorname{diag}(n^{-1},n,1,1)$ is
two dimensional:
  $$\lim_{n\to\infty}c_n H c_n^{-1}=\left(\begin{array}{cc}
 P(s) & 0 \\
 0 & R(t)
 \end{array}\right),$$ where $t,s$ are independent parameters. The group $H$ is a one-parameter subgroup of ${\mathbb R}\times S^1$ that looks like
 a helix,
 and conjugating by $c_n$ coils the helix more tightly. The
{\em algebraic limit} is the one-dimensional group
$${a\text{-}\lim_{k\to\infty}}P_k H P_k^{-1}=\left(\begin{array}{cc}
 I & 0 \\
 0 & R(t)
 \end{array}\right)$$
 because the limit of Lie algebras is described by:
 $$\begin{pmatrix} 0 & n^{-1} t & & \\ 0 & 0 & & \\ & & 0 & -t \\ & & t & 0\end{pmatrix} \xrightarrow[n \to \infty]{} \begin{pmatrix} 0 & 0 & & \\ 0 & 0 & & \\ & & 0 & -t \\ & & t & 0\end{pmatrix}.$$
 In fact, the {\em local geometric limit} is also strictly smaller than the geometric limit; it coincides with the algebraic limit. This is because every element with non-trivial entries in the $P(s)$ block of the geometric limit comes from a sequence of elements of $H$ which go to infinity.

\item[(5)] To construct a limit of a connected group with infinitely many components, consider again the group from (4):
$$H=\left(\begin{array}{cc}
 P(t) & 0\\
 0 & R(t)   \end{array}\right).$$
 The {\em geometric limit} under conjugacy by the sequence $c_n = \operatorname{diag}(1,1,n,n^{-1})$ 
is
$$L=\left\{\left(\begin{array}{cc}
 P(N \pi) & 0\\
 0 &   (-1)^N P(t)  \\
\end{array}\right):\quad N\in{\mathbb Z},\quad t\in{\mathbb R}\right\}$$
and this has countably many components.

\item[(6)] Next, here is an example where the conjugacy limit of a group is a proper subgroup of itself.
Consider the group $L$ from (5). Now conjugate $L$ by the sequence $c_n=\operatorname{diag}(n,n^{-1},1,1)$. The limit is:
$$L' =\left(\begin{array}{cc}
 I_2 & 0\\
 0 & P(t)\end{array}\right).$$

\item[(7)] The following subgroup of $\GL(6,\RR)$ has infinitely many non-conjugate {\em geometric limits}. Fix $\alpha$ and define
$$H=H(s,t)=\left(\begin{array}{ccccc}
 e^s & 0 & 0 & 0 & 0\\
 0 & e^t & 0 & 0 & 0\\
 0 & 0 & 1 & s & t\\
 0 & 0 & 0 & 1 & 0\\
 0 & 0 & 0 & 0 & 1 \end{array}\right)\qquad 
 c_n=\left(\begin{array}{ccccc}
 1 & 0 & 0 & 0 & 0\\
 0 & 1 & 0 & 0 & 0\\
 0 & 0 & 1 & 0 & 0\\
 0 & 0 & 0 & n & \alpha n\\
 0 & 0 & 0 & n & n^{-1}+\alpha n \end{array}\right)$$
Then
 $$c_n H c_n^{-1}=
 \left(\begin{array}{ccccc}
 e^s & 0 & 0 & 0 & 0\\
 0 & e^t & 0 & 0 & 0\\
 0 & 0 & 1 & n^{-1}+n(\alpha s-t) & -n(\alpha s-t)\\
 0 & 0 & 0 & 1 & 0\\
 0 & 0 & 0 & 0 &1 \end{array}\right)$$
 The limit as $n\to\infty$ is the two-dimensional group
 $$
 L_{\alpha}=L_{\alpha}(s,u)=\left(\begin{array}{ccccc}
 e^s & 0 & 0 & 0 & 0\\
 0 & e^{\alpha s} & 0 & 0 & 0\\
 0 & 0 & 1 & u & -u\\
 0 & 0 & 0 & 1 & 0\\
 0 & 0 & 0 & 0 &1 \end{array}\right)$$
Consideration of the character shows that
if $\alpha\ne\beta$ then $L_{\alpha}$ is not conjugate to $L_{\beta}$.

\item[(8)] We modify the previous example to obtain algebraic limit groups.
For $\beta\in{\mathbb R}$ define a two-dimensional representation of $\RR^2$ by
$$\sigma_{\beta}:{\mathbb R}^2\longrightarrow \SL(2,{\mathbb R})\qquad \sigma_{\beta}(x,y)=\left
(\begin{array}{cc} 1 &  \beta x-y \\
0 & 1 \\
\end{array}\right),$$
and define also the two-dimensional representation of $\RR^2$, by 
$$\tau(s,t) = \begin{pmatrix} 1 & s & t\\ 0 & 1 & 0\\ 0 & 0 & 1\end{pmatrix}.$$
Now define a three-parameter algebraic subgroup $H(r,s,t)$ of $\SL(11,\RR)$ as a direct sum of representations
  $$H(r,s,t)=\sigma_0(r,s)\oplus\sigma_1(r,s)\oplus\sigma_2(r,s)\oplus \sigma_1(t,r)\oplus\tau(s,t).$$
 
Let $c_n = Id_9\oplus b_n$, where $b_n = \begin{pmatrix} n & \alpha n\\ n & n^{-1} + \alpha n \end{pmatrix}$. Then the limit of $H$ under conjugacy by $c_n$ is the algebraic group
$$L_{\alpha}(r,s,u)= \sigma_0(r,s)\oplus\sigma_1(r,s)\oplus\sigma_2(r,s)\oplus \sigma_1(\alpha s,r)\oplus\tau(u,-u).$$
We claim that the function which sends $\alpha\in{\mathbb R}$ to
the conjugacy class of $L_{\alpha}$ is finite to one.
This requires an invariant. 

Given  a unipotent representation
$\rho:{\mathbb R}^N\longrightarrow \SL(k,{\mathbb R})$ and $0\ne x\in{\mathbb R}^N$,  
the nullity of $(\rho(x)-Id)$ only depends
on $[x]\in{\mathbb R}P^{N-1}$. This defines a function $\mathcal N_{\rho}:{\mathbb R}P^{N-1}\longrightarrow \{0,1,\cdots,k\}$.
If $\rho$ and $\rho'$ have conjugate images there is a projective transformation $T\in \PGL(N,{\mathbb R})$
such that $\mathcal N_{\rho'}=\mathcal N_{\rho}\circ T$. 

Thinking of $L_\alpha$, in the above example, as a representation of $\RR^3$ in terms of the parameters $r,s,u$, we have that $\mathcal N_{L_{\alpha}}([r:s:u])=8$ iff $u=0$
and $[r:s]$ is one of four points $$[1:0],\  [1:1],\  [1:2],\  [1:\alpha]$$ on the projective line
$[*:*:0]$ in ${\mathbb R}P^2$. The cross ratio of $\{0,1,2,\alpha\}$, up to the action of the finite group that permutes these points, provides an invariant which shows there is a continuum of non-conjugate $L_\alpha$.
\end{enumerate}

\subsection{Properties of limits}
\label{sec:limit-props}

As we saw in the previous section, it is possible for the dimension of a subgroup to increase under taking limits. However, in the algebraic setting, this does not happen. Therefore, in this setting, the connected component of the geometric limit (the connected geometric limit) is equal to the Lie algebra limit (the group generated by taking the limit at the Lie algebra level first and then exponentiating).

\begin{proposition}\label{prop:limit-dimension} 
Let $G$ be an algebraic group (defined over $\CC$ or $\RR$). Suppose that $H$ is an algebraic subgroup and $L$ a conjugacy limit of $H$. Then $L$ is algebraic and $\dim L=\dim H$.
\end{proposition}

\begin{proof} Suppose $c_n\in G$ and $H_n=c_n^{-1}Hc_n$ converges in the
Chabauty topology to a subgroup $L$. Assume, for contradiction, that $\dim L>\dim H$. Then for every neighborhood
$U$ of the identity in $G$ the number of connected components of
$U\cap H_n$ goes to infinity as $n \to \infty$. Let $V$ be a variety of dimension $\dim G-\dim H$ passing through the identity, and smooth there. Choose $V$ so
that it is transverse to each $H_n$ (for $n$ sufficiently large) in a fixed neighborhood $U$ of the identity. The set $V\cap H_n \cap U$ is a finite set of points, with cardinality going to infinity as $n$ goes to infinity. However, this is impossible, because the degree of the variety $V$ is constant and the degree of the varieties $H_n$ is also constant equal to that of $H$. Therefore the degree of $V \cap H_n$ is bounded and the cardinality of the finite sets $V \cap H_n \cap U$ must also be bounded.

In the case $G$ is defined over $\CC$, the fact that the limit group $L$ is algebraic follows from a theorem of Tworzewski and Winiarski \cite{ Tworzewski} using work of Bishop \cite{bishop}. They showed that the set of all pure dimensional algebraic subsets of ${\mathbb C}^n$ of bounded degree is compact in the topology of local uniform convergence. From this
one may deduce the 
case $G$ is defined over $\RR$ by complexification.

\end{proof}
\noindent
 Note that Proposition~\ref{prop:limit-dimension} applies also when $G$ is the connected component of an algebraic group. In particular Proposition~\ref{prop:limit-dimension} holds when $G$ is the general or special linear group over $\RR$ or $\CC$, and also when $G = \PGL(n,\RR)$, which are the main cases of interest in this article.
 
Next, we investigate the behavior of multiple limits taken in succession, with the goal of showing that the relation ``$L$ is a limit of $H$" induces a partial order of the space of closed subgroups of $G$.
To begin, we study the behavior of the normalizer. Given a closed subgroup $H$ of $G$, let $H_0$ denote its identity component. The normalizer of $H_0$ in $G$ will be denoted $N_G(H_0)$.

\begin{proposition}\label{prop:normalizer-dim}Let $G$ be an algebraic Lie group (defined over $\CC$ or $\RR$), let $H$ be an algebraic subgroup and let $L$ be any limit of $H$. Then $\dim N_G(H_0)\le\dim N_G(L_0)$ with
equality if and only if $L$ and $H$ are conjugate.
\end{proposition}
\begin{proof}
Let $\frakh$ and $\frakl$ denote the Lie algebras of $H$ and $L$ respectively. Then the normalizers $N_G(H_0)$ and $N_G(L_0)$ are equal to the normalizers of the respective Lie algebras $N_G(\frakh)$ and $N_G(\frakl)$.
By Proposition~\ref{prop:limit-dimension}, $\dim \frakh = \dim \frakl =: k$. So $\frakh$ and $\frakl$ define points in the projectivization $\mathbb P V$ of the $k^{th}$ exterior power $V = \Lambda^k \frakg$.
The orbit $G \cdot \frakh$ under the adjoint action of $G$ is a smooth subset of $\mathbb P V$ corresponding to the Lie algebras of conjugates of $H$.  The closure (in the classical topology) $\overline{G \cdot \frakh}$, which contains $\frakl$, is a union of $G \cdot \frakh$ and orbits of strictly smaller dimension (in the case that $G$ is defined over $\RR$, this follows because the orbit ${G \cdot \frakh}$ is semi-algebraic by the Tarski-Seidenberg theorem). 
 If $\frakl \in G \cdot \frakh$, then $L$ is conjugate to $H$ and, of course, $N_G(L_0)$ is conjugate to $N_G(H_0)$. Assume then that $\frakl \in \overline{G \cdot \frakh} \setminus G \cdot \frakh$, so that $\dim G \cdot \frakl < \dim G \cdot \frakh$. Then
\begin{align*}
\dim N_G(\frakl) &= \dim G - \dim G \cdot \frakl \\ 
& > \dim G - \dim G \cdot \frakh = \dim N_G(\frakh).
\end{align*}
\end{proof}

Let $\Grp(G)$ denote the set of conjugacy classes of connected, algebraic,
 Lie subgroups of an algebraic Lie group $G$. If $L$ is the connected geometric limit of $H$ under some sequence of conjugacies (so $L$ is the identity component of a limit of $H$), then
we write $H\rightarrow_0 L$. 

\begin{theorem}\label{thm:partial-order} Let $G$ be an algebraic Lie group. The relation of being a connected geometric limit induces a partial order on the connected, algebraic, sub-groups $\Grp(G)$. Moreover
the length of every chain is at most $\dim G$.
\end{theorem}
\begin{proof} It follows from Proposition~\ref{prop:normalizer-dim} that $H\rightarrow_0 L$ and $L\rightarrow_0 H$ implies $L$ is 
conjugate to $H$.
It remains to show transitivity. Suppose $\lim H_n= L$ and $K=\lim L_m$ with 
$H_n=a_nHa_n^{-1}$ and $L_m=b_mLb_m^{-1}$. 
Let $\mathcal C(G)$ denote the closed subgroups of $G$, equipped with the Chabauty topology (i.e. the topology of Hausdorff convergence in compact sets). The map $\theta_m:{\mathcal C}(G)\longrightarrow{\mathcal C}(G)$ given by $\theta_m(P)=b_mPb_m^{-1}$ is continuous.
Hence $\lim_{n\to\infty}\theta_m(H_n)=L_m$ in ${\mathcal C}(G)$. It follows there is a sequence $\theta_{m_n}(H_n)$
which converges to $K$ in ${\mathcal C}(G)$.
\end{proof}

Now we restrict our attention to the case when  $G $ is locally isomorphic to $\GL_n \RR$. If $L$ is a limit of $H$, the eigenvalues of elements of $L$ are related to those of $H$. 
This leads to an obstruction to $L$ being a limit of $H$. 
The idea is that
 under degeneration eigenvalues either are unchanged or degenerate.
  
An element $A$ of $\frakgl_N = \operatorname{End}(\RR^n)$ has a well defined
characteristic polynomial, denoted $\charpoly(A)$.
Given a Lie sub-algebra $\frakh$ of $\frakgl$, we denote by $\charpolys(\frakh)$ the closure of the subset
of ${\mathbb R}[x]$ consisting of characteristic polynomials of all elements in $\mathfrak{h}$. 
Thus $\charpolys(\frakh)$ is closed and invariant under conjugation of $\frakh$.

\begin{proposition}
\label{prop:charpoly}
 Suppose $H$ is a closed algebraic subgroup of $\GL_n \RR$, and $L$ is a conjugacy limit of $H$. Then 
$\charpolys(\frakl)\subset \charpolys(\frakh)$, where $\frakh, \frakl \subset \frakgl(n)$ denote the Lie algebras of $H$ and $L$ respectively.
\end{proposition}
\begin{proof} Suppose $p(x) = \charpoly(\ell)$ for some $\ell \in \frakl$. By assumption we have that $c_k H c_k^{-1} \to L$ and since $H$ is an algebraic subgroup, we have convergence at the Lie algebra level as well: $\Ad_{c_k} \frakh \to \frakl$. Hence, there exists a sequence $h_{k} \in \frakh$ so that $\Ad_{c_k} h_{k} \to \ell$ as $k \to \infty$. The characteristic polynomials  $\charpoly(h_{k}) = \charpoly(\Ad_{c_k} h_{k})$ then converge to $\charpoly(\ell)$ and therefore $\charpoly(\ell) \in \charpolys(\frakh)$ since $\charpolys(\frakh)$ is closed. 
\end{proof}
\noindent
For example, if $P(t)$ and $R(t)$ are the one parameter subgroups of $\GL_2\RR$ described in Section~\ref{sec:examples}, and $\mathfrak{p}$ and $\mathfrak{r}$ denote their respective Lie algebras, then $\charpolys(\mathfrak{r})=\{x^2 + \theta^2: \theta \in \RR\}$ while $\charpolys(\frakp) = \{x^2\}$. Proposition~\ref{prop:charpoly} implies that $R(t)$ is not a limit of $P(t)$. This also follows from Theorem~\ref{thm:partial-order} because $P(t)$ is a (connected) limit of $R(t)$ and they are not conjugate (see example (2) from Section~\ref{sec:examples}).

We conclude this section by applying Proposition~\ref{prop:charpoly} to prove that $\HH^2 \times \RR$ geometry is not contained in any limit of hyperbolic geometry, when both are considered as sub-geometries of projective geometry. A more general statement on which Thurston geometries can arise as limits of hyperbolic geometry  will be given in Theorem~\ref{cor:Thurston-geoms_intro}.
\begin{proposition}\label{prop:H2timesR}
$\Isom_+({\mathbb H}^2\times{\mathbb R})$ is not a subgroup of a
limit of $\PO(3,1)$ in $\GL(4,{\mathbb R})$ and therefore $\HH^2 \times \RR$ is not a sub-geometry of any limit of hyperbolic geometry inside of projective geometry.\end{proposition}

\begin{proof}
The (almost) embedding of $\HH^2 \times \RR$ geometry in $\RP^3$ geometry represents the isometries of $\HH^2 \times \RR$ which preserve the orientation of the $\RR$ direction as
$$\Isom_+({\mathbb H}^2\times{\mathbb R})=\left\{\left(\begin{array}{cc} e^t A & 0\\ 0 & e^{-3t}
\end{array}\right):\quad A\in SO(2,1), t \in \RR\right\}.$$
The Lie algebra is described by 
$$ \mathfrak{isom}_+(\HH^2 \times \RR) = \left\{\left(\begin{array}{cc} t I_3 + a & 0\\ 0 & -3t \end{array}\right):\quad a\in \frakso(2,1), t \in \RR \right\}.$$
 The eigenvalues of elements of this Lie sub-algebra of $\frakgl(4)$ are of the form  $t,t+\lambda,t-\lambda,-3t$ with $\lambda^2\in{\mathbb R}$.
The set of characteristic polynomials is:
$$\charpolys( \mathfrak{isom}_+(\HH^2 \times \RR)) = \{ (x-t)((x-t)^2-\lambda^2)(x+3t):t,\lambda^2\in{\mathbb R}  \}.$$
 On the other hand, the isometries of $\HH^3$ in the projective model are $\PO(3,1)$ and
 $$\charpolys(\frakso(3,1))=\{(x^2-\lambda^2)(x^2+\theta^2):\ \lambda,\theta\in{\mathbb R}\}.$$
 Inspection of these two sets and an application of Proposition~\ref{prop:charpoly} proves the claim.
\end{proof}

\section{Symmetric subgroups}
\label{sec:symmetric-subgroups}

We turn now to a special class of Lie groups and their subgroups, namely semi-simple Lie groups $G$ with finite center and their symmetric subgroups $H$. We will give an explicit description of the conjugacy limits of such symmetric subgroups and a more descriptive version of Theroem~\ref{thm:symmetric-limits} in Section~\ref{sec:symm_general}. 
Then, Sections~\ref{sec:symmetric-PGL}--\ref{sec:Opq} are dedicated to symmetric subgroups of the (projective) general linear group and their limits.

\subsection{Symmetric subgroups in a semi-simple Lie group}\label{sec:symm_general}
Let $G$ be a connected semi-simple Lie group of non-compact type and with finite center.
Let $\sigma: G \to G$ be an involutive automorphism, i.e. $\sigma$ is a continuous automorphism with $\sigma^2 = 1$. 
The subset of
fixed points
$$G^\sigma =\{g\in G\ :\ \sigma(g)=g\ \}$$
is a closed subgroup of $G$. 
A closed group $H$ with $G^\sigma_0 \subset H \subset G^\sigma$, where $G^\sigma_0$ denotes the connected component of the identity, is called a \emph{symmetric subgroup} of $G$. 
The quotient space $G/H$ is an \emph{affine symmetric} space. 
Let us give some examples of affine symmetric spaces: 
\begin{itemize}
\item When $H = K$ is a maximal compact subgroup, then $G/K$ is a Riemannian symmetric space. 
\item When $G = L\times L$ and $H = \mathrm{Diag}(L)$ is the diagonal, then $G/H\cong L$, via $(l_1,l_2)\mapsto {l_1l_2^{-1}}$.
\item Let $G = \PO(p,q)$ and $H = \OO(p-1,q)$. Then the affine symmetric space $$\XX(p,q) = \PO(p,q)/\OO(p-1, q)$$ is a model space for semi-Riemannian manifolds of signature $(p-1,q)$ and of constant curvature $-1$. 
The geometries $(\XX(p,q), \PO(p,q))$ are in fact subgeometries of real projective geometry $(\RP^{p+q-1}, \PGL_{p+q} \RR)$, and we will describe their limiting geometries explicitly in Section~\ref{sec:pfqf}.
\item The symmetric subgroups of $\PGL_n\RR$ are $\operatorname{P}(\GL_p\RR\oplus \GL_q\RR)$ and $\PO(p,q)$ where 
$p+q = n$, or $\operatorname{P}(\GL_m\CC)$ and $\operatorname{P}(\Sp(2m,\RR))$ where $n = 2m$, where for a subgroup $H' \subset \GL_n \RR$, $\operatorname{P}(H')$ denotes the image of $H'$ under the projection $\GL_n \RR \to \PGL_n \RR$. See Section~\ref{sec:symmetric-GLn}.

\end{itemize}

In order to describe the conjugacy limits of symmetric subgroups, we will make use of the rich structure theory of affine symmetric spaces and symmetric subgroups. In order to keep the presentation concise we recall only the necessary details of the structure theory and refer the reader for more details and proofs to \cite{Schlichtkrull, Rossmann, nomizu, MB}. 
We denote by $\frakg$ and $\frakh$ the Lie algebra of $G$ and $H$ respectively, and let the differential of $\sigma$, an involution of $\frakg$, be again denoted by $\sigma: \frakg \to \frakg$. Then $\frakh$ is the $+1$ eigenspace of $\sigma$ and we denote the $-1$ eigenspace by $\frakq$.
This gives the orthogonal decomposition $$\frakg = \frakh \oplus \frakq.$$ 
Note also that $[\frakh, \frakq] \subset \frakq$, $[\frakq,\frakq] \subset \frakh$.

There exists a Cartan involution $\theta: \frakg \to \frakg$, which commutes with $\sigma$. We denote by $K = G^\theta$ the maximal compact subgroup of $G$ given by the fixed points of $\theta$  and we let $\frakg = \frakk \oplus \frakp$ be the corresponding Cartan decomposition of $\frakg$. 
Since the involutions commute, all four subspace are preserved by both involutions and so is the following decomposition:
$$
\frakg = \frakk \cap \frakh \oplus \frakk \cap \frakq \oplus \frakp \cap \frakh \oplus \frakp \cap \frakq. 
$$

Next, we may choose a maximal abelian sub-algebra $\fraka \subset \frakp$ so that the intersection $\frakb = \fraka \cap \frakq$ is a maximal abelian sub-algebra of $\frakp \cap \frakq$. Note that $\frakb$ is unique up to the action of $H \cap K$. We let $A := \exp (\fraka)$ be the corresponding connected subgroup of $G$ and $B := \exp(\frakb) \subset A$. Note that in some cases $\frakb = \fraka$ while in others the containment is strict. The following is well-known (see for example Proposition 7.1.3 of \cite{Schlichtkrull}):

\begin{theorem}[$KBH$ decomposition]\label{thm:KAH}
For any $g \in G$ there exists $k \in K$, $b\in B$ and $h \in H$, such that 
$g = kbh$. Moreover $b$ is unique up to conjugation by the Weyl group $W_{H \cap K} := N_{H\cap K} (\frakb)/ Z_{H\cap K}( \frakb)$. 
\end{theorem}

\begin{remark}
This factorization theorem will be our main tool in determining the limits of $H$. 
In the case $G = \PGL_n \RR$, this is equivalent to a matrix decomposition theorem for $\GL_n \RR$. 
Furthermore, in this case we may conjugate so that $A$ is a subgroup of diagonal matrices. Hence every conjugacy limit of a symmetric subgroup $H <  \PGL_n \RR$ is conjugate to a conjugacy limit by a sequence of diagonal matrices.
 \end{remark}

\noindent Using Theorem~\ref{thm:KAH}, we prove Theorem~\ref{thm:symmetric-limits} from the introduction, restated here for convenience:
\begin{reptheorem}{thm:symmetric-limits}
Let $H$ be a symmetric subgroup of a semi-simple Lie group $G$ with finite center. 
Then for any geometric limit $L'$ of $H$ in $G$, there exists $X \in \frakb$, so that $L'$ is conjugate to the limit group $L= \lim_{t \to \infty} \exp(tX) H \exp(-tX)$ obtained by conjugation by the one parameter group generated by~$X$. Furthermore, 
$$L = Z_H(X) \ltimes N_+(X),$$ 
where $Z_H(X)$ is the centralizer in $H$ of $X$, and $N_+(X)$ is the connected nilpotent subgroup
$$ N_+(X) := \{g \in G : \lim_{t\to \infty} \exp(tX)^{-1} g \exp(tX) = 1 \}.$$
\end{reptheorem}

\begin{proof}
Let $(c_n) \subset G$ be a sequence and $L = \lim_{n \to \infty} c_n H c_n^{-1}$. By Theorem~\ref{thm:KAH}, we may factorize  $c_n = k_n b_n h_n$, where $k_n \in K, b_n \in B, h_n \in H$. 
Then, 
\begin{align*}
c_n H c_n^{-1} &= k_n b_n h_n H h_n^{-1} b_n^{-1} k_n^{-1} \\
&= k_n b_n H b_n^{-1} k_n^{-1}
\end{align*}
We may assume, after passing to a subsequence, that $k_n \to k \in K$, so it follows that $c_n H c_n^{-1}$ converges if and only if $b_n H b_n^{-1}$ converges and their limits are conjugate by $k$.
Hence we assume that $k_n = 1 = h_n$ and consider only conjugacies by sequences $(b_n) \in B$.

Consider the set of roots 
$$\Sigma(\frakg, \fraka):=\{ \alpha \in \fraka^*\, |\, \text{ there exists non-zero } Z\in \frakg \text{ with }  \ad(X)(Z) = \alpha(X) Z , \, \text{  for all } X \in \fraka\}.$$ Denote by $\frakg_\alpha:= \{Z\in \frakg\, |\, \ad(X)(Z) = \alpha(X) Z , \, \text{  for all } X \}$ the root spaces and let $\frakg = \sum_{\alpha \in \Sigma(\frakg, \fraka)} \frakg_\alpha$ be the corresponding root space decomposition of $\frakg$. Choose a basis for $\frakg$ compatible with this root space decomposition. We work with the adjoint representation $\Ad: G \to \GL(\frakg) \cong \GL(N)$, expressed in this basis, which takes the subgroup $B$ to a subgroup of diagonal matrices in $\GL(N)$. The diagonal entries $\Ad(b_n)_{ii}$ of $\Ad(b_n)$ are positive, and we may assume they are arranged in increasing order: $ \Ad(b_n)_{jj} \geq \Ad(b_n)_{ii}$ for all $n$ and $j > i$. Further, we may assume that for consecutive indices $j = i+1$, the  diagonal entries of $\Ad(b_n)$ satisfy exactly one of the following:
\begin{itemize}
\item Either $\Ad(b_n)_{jj} = \Ad(b_n)_{ii}$ holds for all $n$, or
\item $\Ad(b_n)_{ii} < \Ad(b_n)_{jj}$ holds for all $n$, and $\Ad(b_n)_{jj}/\Ad(b_n)_{ii} \to \infty$.
\end{itemize}
For if $\Ad(b_n)_{jj}/\Ad(b_n)_{ii}$ remains bounded we may multiply $(b_n)$ by a sequence $(b_n') \subset B$ which remains in a compact subset of $B$ so that the above holds; the resulting limit differs only by conjugation (by the limit of $b_n'$).

Now consider a sequence $(h_n) \subset H$, so that $b_n h_n b_n^{-1} \to \ell \in L$. Then, the matrix entries 
$$\Ad(b_n h_n b_n^{-1})_{ji} = \frac{\Ad(b_n)_{jj}}{\Ad(b_n)_{ii}} \Ad(h_n)_{ji}$$
converge to $\Ad(\ell)_{ji}$ as $n \to \infty$. We therefore have that, for $i \leq j$,
\begin{equation}
\left \{ \begin{array}{ll}
 \Ad(h_n)_{ji} \longrightarrow \Ad(\ell)_{ji} & \text{ if $\Ad(b_n)_{ii} = \Ad(b_n)_{jj}$, or}\\
 \Ad(h_n)_{ji} \longrightarrow 0 & \text{ if $\frac{ \Ad(b_n)_{jj}}{\Ad(b_n)_{ii} }\to \infty$ as $ n \to \infty$.}
 \end{array}\right .\end{equation}
 This gives us information about the block lower diagonal entries of $\Ad(h_n)$. We use the involution $\sigma$ to obtain control over the block upper diagonal matrix entries. 
 Note first that the involution $\sigma$ satisfies $\Ad(\sigma(b_n h_n b_n^{-1}))_{ji} =  \Ad(b_n^{-1} h_n b_n)_{ji}$. 
 Now $\sigma(b_n h_n b_n^{-1}) \to \sigma(\ell)$ as $n \to \infty$ and therefore the matrix entries 
  \begin{align*}
\Ad(\sigma(b_n h_n b_n^{-1}))_{ji} &=  \Ad(b_n^{-1} h_n b_n)_{ji}\\
&= \frac{\Ad(b_n)_{ii}}{\Ad(b_n)_{jj}} \Ad(h_n)_{ji}
\end{align*}
also converge. Then, similar to the above, we conclude that, for $i \geq j$:
\begin{equation}
\left \{ \begin{array}{ll}
 \Ad(h_n)_{ji} \longrightarrow \Ad(\ell)_{ji} & \text{ if $\Ad(b_n)_{ii} = \Ad(b_n)_{jj}$, or}\\
 \Ad(h_n)_{ji} \longrightarrow 0 & \text{ if $\frac{ \Ad(b_n)_{jj}}{\Ad(b_n)_{ii} }\to 0 $ as $ n \to \infty$.}
 \end{array}\right .\end{equation}
 Therefore, we conclude that $\Ad(h_n)$ converges, and because $G \to \Ad(G)$ is proper (because $Z(G)$ is finite), we may take a subsequence so that $h_n \to h_\infty$ as $n \to \infty$, where
\begin{equation}
\left \{ \begin{array}{ll}
 \Ad(h_\infty)_{ji} = \Ad(\ell)_{ji} & \text{ if $\Ad(b_n)_{ii} = \Ad(b_n)_{jj}$.}\\
 \Ad(h_\infty)_{ji} = 0 & \text{ if $\Ad(b_n)_{ii} \neq \Ad(b_n)_{jj}$.}
 \end{array}\right .\end{equation}
 In other words, $\Ad(h_\infty)$ belongs to the centralizer $Z_{\GL(N)}(\Ad(b_n))$. It follows that $$h_\infty \in Z_H(b_n).$$
 This is because the commutator of $h_{\infty}$ with $b_n$ is in the kernel of $\Ad$.  However $b_n$ lies in a one parameter
 group that commutes with $h_{\infty}$ so their commutator lies in a connected, hence trivial, subgroup of the center of $G$. In the above expressions, note that $Z_{\GL(N)}(\Ad(b_n))$ and $Z_H(b_n)$ are independent of $n$ by our assumptions on the sequence $(b_n)$.  
 
 Observe that $Z_H(b_n) \subset L$. In fact, we have shown that $L$ is an \emph{expansive limit} of $H$ about the subgroup $Z_H(b_n)$ (see Section~\ref{sec:other-notions}).
 From the above arguments, we may also conclude that
\begin{align*}
\Ad(\ell)_{ji} &= \lim_{n \to \infty} \frac{\Ad(b_n)_{jj}}{\Ad(b_n)_{ii}} \Ad(h_n) = 0
\end{align*} in the case that $\Ad(b_n)_{jj}/\Ad(b_n)_{ii} \to 0$. It then follows that $\Ad(b_n^{-1} \ell b_n) \to \Ad(h_\infty)$ because
\begin{equation}
\Ad(b_n^{-1}\ell b_n)_{ji} = \frac{\Ad(b_n)_{ii}}{\Ad(b_n)_{jj}}\Ad(\ell)_{ji} 
\left \{ \begin{array}{ll}
 = \Ad(h_\infty)_{ji} & \text{ if $\Ad(b_n)_{ii} = \Ad(b_n)_{jj}$}\\
  = 0 = \Ad(h_\infty)_{ji} & \text{ if $\Ad(b_n)_{jj} < \Ad(b_n)_{ii}$}\\
  \longrightarrow 0 = \Ad(h_\infty)_{ji} & \text{ if $\Ad(b_n)_{jj} > \Ad(b_n)_{ii}$}
 \end{array}\right .\end{equation}
 
 \noindent Hence $b_n^{-1} \ell b_n \to h_\infty' \in Z_H(b_n)$ (in fact, $h_\infty' = h_\infty$, but we have only shown $\Ad(h_\infty) = \Ad(h_\infty')$).
 Therefore $b_n^{-1}  (h_\infty')^{-1} \ell b_n \to 1$ as $n \to \infty$ and therefore $(h_\infty')^{-1} \ell$ lies in the group $$ N_+ := \{ g \in G: b_n^{-1} g b_n \to 1 \text{ as } n \to \infty \}.$$ 
It follows that 
\begin{equation}
\label{eqn:contained}
L \subset \langle Z_H(b_n), N_+ \rangle = Z_H(b_n) \ltimes N_+.
\end{equation}
We next show equality. We have that $L = Z_H(b_n) \ltimes N'$ where $N' = N_+ \cap L \subset N_+$ is a closed subgroup. We also have that $N_+$ is connected because $N_+$ is preserved by conjugation by $B$, and conjugation by $b_n^{-1}$, for large $n$, brings elements arbitrarily close to the identity; hence all elements are in the identity component. So equality in \eqref{eqn:contained} will follow by showing that $\dim N_+ + \dim Z_H((b_n)) = \dim H$. We show this at the Lie algebra level. Observe that $\mathfrak g=\mathfrak z +\mathfrak n_++\mathfrak n_-$ where
\begin{align*}
\mathfrak z &= \{ Y \in \frakg: \Ad(b_n) Y = Y \text{ for all } n\}\\
\mathfrak n_+ &= \{ Y \in \frakg: \Ad(b_n^{-1})Y \to 0 \text{ as } n \to \infty\}\\
\mathfrak n_- &= \{Y \in \frakg: \Ad(b_n)Y \to 0 \text{ as } n \to \infty\}
\end{align*}

The involution $\sigma$ fixes $\mathfrak z$, and exchanges $\mathfrak n_+$ and $\mathfrak n_-$. Therefore $\mathfrak n_+$ and $\mathfrak n_-$ have the same dimension and the dimension of the $+1$ eigenspace $\frakh$ of $\sigma$ is equal to $\dim \mathfrak z + \dim \mathfrak n_+$ while the dimension of the $-1$ eigenspace is $\dim \mathfrak n_+$.  Since $\mathfrak z$ is the Lie algebra of $Z_H(b_n)$ and $\mathfrak n_+$ is the Lie algebra of $N_+$, we conclude that $L = Z_H(b_n) \ltimes N_+$. Observe that $\Ad(N_+)$ is upper triangular and therefore $N_+$ is nilpotent. 

The proof now concludes by observing that any sequence $(b_n) \subset B$ whose eigenvalues satisfy our above assumptions will produce exactly the same limit. In particular, let $X = \log(b_m)$ for some $m$. Then conjugating $H$ by the sequence $b_n' = \exp(t_n X)$, also produces $L$ as the limit for any sequence of reals $(t_n)$ such that $t_n \to \infty$. In this case $Z_H(\exp(t_n X)) = Z_H(X)$ and this implies the claim.
\end{proof}

The proof of Theorem~\ref{thm:symmetric-limits} shows that the limit is determined up to conjagacy by 
an ordered partition of the numbers $\{1,\ldots, N\}$. Hence
 there are finitely many limits, up to conjugacy, of a symmetric subgroup $H$ of $G$. We now introduce some notation in order to enumerate these limits.
Consider the root system $\Sigma(\frakg, \frakb) \subset \frakb^*$ defined by the adjoint action of $\frakb$ on $\frakg$. The Weyl group $W = N_K(\frakb)/Z_K(\frakb)$ acts on $\frakb$; it is the group generated by reflections in the hyperplanes determined by the roots.
We may similarly consider the adjoint action of $\frakb$ on the smaller Lie algebra $\frakg_{\tau} := \frakh \cap \frakk \oplus \frakp \cap \frakq$; this is the Lie algebra of the symmetric subgroup defined by the involution $\tau = \sigma \theta$. Then the corresponding root system $\Sigma(\frakg_{\tau}, \frakb)$ is contained in $\Sigma(\frakg, \frakb)$, and the corresponding Weyl group $W_{H \cap K} := N_{H \cap K}(\frakb)/Z_{H \cap K}(\frakb)$ is a subgroup of $W$. 
 Let $\Sigma^+ \subset \Sigma(\frakg, \frakb)$ denote a system of positive roots, and let $\Delta \subset \Sigma^+$ be a choice of simple roots. Define: 
 \begin{align*}
 \frakb^+ &= \{ Y \in \frakb : \alpha(Y) > 0 \text{ for all } \alpha \in \Delta\}, & B^+ &= \exp{\frakb^+},\\
 \overline{\frakb^+} &= \text{ the closure of } \frakb^+, & \overline{B^+} &= \exp{\overline{\frakb^+}}.
 \end{align*}
 Then $\overline{\frakb^+}$ is the closed Weyl chamber corresponding to $\Sigma^+$; it is (the closure of) a fundamental domain for the action of $W$. A closed Weyl chamber for the action of $W_{H \cap K}$ is then given by a union $\mathcal W \cdot \overline{\frakb^+}$ of translates of $\overline{\frakb^+}$ by a set $\mathcal W$ of coset representatives for the quotient $W/W_{H \cap K}$. 

By Theorem~\ref{thm:KAH}, any element $g \in G$ may be decomposed as $g = kbh$, where $k \in K, b\in B, h \in H$. There exists $w_1 \in W_{H \cap K}$ so that the conjugate $w_1^{-1} b w_1$ lies in the exponential image of the closed Weyl chamber $\mathcal W \cdot \overline{\frakb^+}$. Therefore $w_1^{-1} b w_1 = w^{-1} b^+ w$ where $b^+ \in \overline{B^+}$ and $w \in \mathcal W$. Therefore, we may write $$g = k'b^+wh'$$ where $k' = k w_1 w^{-1}$ and $h' = w_1^{-1} h$. Here $b^+ \in \overline{B^+}$ and $w \in \mathcal W$ are uniquely determined. 

 It follows that any limit of $H$ may be obtained by conjugating by a sequence in $\overline{B^+}w$ for some $w \in \mathcal W$. We next apply these observations in combination with Theorem~\ref{thm:symmetric-limits} to obtain an enumeration of the limit groups in terms of the element $w$ and the behavior of the sequence in $\overline{B^+}$.
Let $I \subset \Delta$ be a subset, and $\Sigma_{I}^+$ the span of $\Delta - I$ in $\Sigma^+$. 
We set 
\begin{align*}
{\frakb}_I &= \bigcap_{\alpha \in I} \ker(\alpha) \subset {\frakb}, &  B_I &= \exp({\frakb}_I),\\
\frakb_I^+ &= \bigcap_{\alpha \in I} \ker(\alpha) \cap \bigcap_{\alpha \in \Delta - I} \{ X \in {\frakb} \,|\, \alpha (X) >0 \}, & B_I^+ &= \exp(\frakb_I^+).
 \end{align*}

\begin{theorem}\label{thm:limit_nice}
Let $H$ be a symmetric subgroup of a semi-simple Lie group $G$ with finite center. Let $L$ be a limit of $H$ under conjugacy in $G$. Then $L$ is conjugate to a subgroup of the form
$$L_{I,w} = Z_{H_w}(\frakb_I) \ltimes N_I$$
where $w \in \mathcal W$ and $I \subset \Delta$ is a subset of the set of simple roots $\Delta \subset \Sigma(\frakg, \frakb)$. Here, $H_w := wHw^{-1}$ and $N_I$ is the connected subgroup of $G$ with Lie algebra $
\frakn_I = \sum_{\alpha \in \Sigma_{I}^+ } \frakg_{\alpha}$. Further, any $L_{I,w}$ is achieved as a limit.
\end{theorem}

\begin{proof}
By Theorem~\ref{thm:symmetric-limits}, we may assume (after conjugating) that $L = Z_H(X) \ltimes N_+(X)$ for some $X \in \frakb$.
Define $$I = \{ \alpha \in \Delta: \alpha(X) = 0\}$$
 and let  $u \in W_{H \cap K}$ and $w \in \mathcal W$ be such that $X' := \Ad(w u)X$ lies in $\overline{\frakb^+}$. Note that, in fact, $X'$ lies in $\frakb_I^+$.
 Then, we have
 \begin{align*}
 Z_H(X) &= Z_H(\Ad(u^{-1} w^{-1}) X') &N_+(X) &= N_+(\Ad(u^{-1}w^{-1}X') \\ &= u^{-1} w^{-1} Z_{H_{wu}}(X') wu & &= u^{-1}w^{-1} N_+(X') wu\\ &= (wu)^{-1} Z_{H_{w}}(X') wu & &= (wu)^{-1} N_+(X') wu.
 \end{align*}
 where in the last step $H_{wu} = H_w$ because $u \in H$. Therefore $L$ is conjugate to $Z_{H_w}(X') \ltimes N_+(X')$. Now, its clear that $N_+(X') = N_I$ because their Lie algebras agree. The Lie algebra of $N_+(X')$ consists of all elements $Y \in \frakg$ such that $\Ad(\exp(-tX'))Y \to 0$ as $t \to \infty$; this is exactly the span of the root spaces $\frakg_\alpha$ for which $\alpha(X') > 0$, in other words $\alpha \in \Sigma_I^+$. It remains to show that $Z_{H_w}(X') = Z_{H_w}(\frakb_I^+)$. The following Lemma will complete the proof. 
  \end{proof} 
 \begin{lemma}\label{lem:centralizers}
 Let $Y_1,Y_2 \in \frakb_I^+$. Then $Z_G(Y_1) = Z_G(Y_2)$.
 \end{lemma}
 \begin{proof}[Proof of Lemma] 
 By replacing $G$ by $\Ad(G) \cong G/Z(G)$, we may assume $G$ is an algebraic subgroup of $\GL(N)$. 
 Let $S_1$ and $S_2$ denote the Zariski closures of the one parameter subgroups generated by $Y_1$ and $Y_2$ respectively. Then $S_1$ and $S_2$ are $\RR$-tori. Therefore $Z_G(Y_1) = Z_G(S_1)$ and $Z_G(Y_2) = Z_G(S_2)$ are Zariski connected closed subgroups of $G$. It follows that $Z_G(Y_1) = Z_G(Y_2)$ since both groups have the same Lie algebra.
 \end{proof}

\begin{corollary}\label{cor:indexlimits}
Let  $\Delta$ be a set of simple positive roots in $\Sigma(\frakg, \frakb)$, let $S$ denote the power set of $\Delta$ and let $\mathcal{W}$ denote a set of representatives of $W/W_{H\cap K}$. Let us denote by $\mathcal{L}(H)$ the conjugacy classes of limits of $H$ in $G$. Then there is a surjection from $S \times \mathcal{W}$ onto $\mathcal{L}(H)$. In particular, there are up to conjugacy only finitely many limits of $H$ in $G$.
\end{corollary}

\begin{remark}
\label{rem:finer-list}
The map of $S \times \mathcal{W}$ onto $\mathcal{L}(H)$ is in general not injective. For example if $I = \emptyset$, then $L_{I,w}$ is conjugate to $L_{I, w'}$ for any $w,w' \in \mathcal W$.
The conjugacy classes of limits can be labeled, with less redundancy, in the following way. First, two subsets $I_1, I_2 \subset \Delta$ give rise to conjugate limit groups $L_{I_1,1}$, $L_{I_2,1}$ if and only if $I_1 = I_2$.
Next, consider a fixed subset $I \subset \Delta$. Let $W_I \subset W$ be the subgroup of the Weyl group which acts trivially on $\frakb_I$. 
Then for any $w \in \mathcal{W}$ and $u \in W_I$, the limit groups $L_{I,w}$ and $L_{I,u w}$ are conjugate.
Let $\mathcal{W}_I$ be a set of representatives of the double cosets $W_I \backslash W /W_{H\cap K}$. For $I = \emptyset$, $W_I$ is trivial, so $\mathcal{W}^{I}= \mathcal{W}$. 
Then the pairs $(I, w)$, where $I \in S$ and $w \in \mathcal W_I$ give all possible limits $L_{I,w}$ of $H$.
We note that this finer enumeration may still have some redundancy; see Remark~\ref{rem:SOpq-redundancy}.
\end{remark}

\begin{remark}
Taking a different point of view, a more detailed analysis of the limits under conjugation can lead to 
 a description of the Chabauty compactification of the affine symmetric space $X= G/H$. 
The Chabauty compactification is defined as follows. Consider the continuous map $\phi: X = G/H \longrightarrow \mathcal{C}(G)$, which sends the left coset $gH$ to the closed subgroup $gHg^{-1}$. Since $\mathcal{C}(G)$, endowed with the Chabauty topology is compact, the closure of $\phi(X)$ defines a compactification of $X$, called the Chabauty compactification \cite{chabauty}. 

In the case when $H=K$ is a maximal compact subgroup the Chabauty compactification of $G/K$ was determined by Guiv\'arch--Ji--Taylor \cite{GJT} and Haettel \cite{Haettel2}. They also show that  the Chabauty compactification of $G/K$ is homeomorphic to the maximal Furstenberg compactification.

The similarities of the above analysis with the definition of the maximal Satake-compactification of $G/H$ as defined in \cite[Theorem 4.10]{Gorodnik_Oh_Shah} suggests that such a homeomorphism might also hold for affine symmetric spaces. 
\end{remark}


\subsection{Limits of symmetric subgroups of $G = \PGL_n\KK$ and $G = \SL_n\KK$}
\label{sec:symmetric-PGL}
We are mainly interested in classifying the limits of sub-geometries of projective geometry whose structure group $H$ is a symmetric subgroup of $G = \PGL_n\KK$, with $\KK= \RR$ or $\CC$. 
So we now apply Theorems~\ref{thm:limit_nice} to the setting of symmetric subgroups in the projective linear group. However, we note that everything described in this section can be easily adapted to the (very similar) case $G = \SL_n \KK$. 
We may choose coordinates so that the Cartan involution $\theta$ that commutes with $\sigma$ is the standard one, i.e.  $\theta(X) = X^{-T}$ for $\PGL_n\RR$ and $\theta(X) = \overline{X}^{-T}$ for $\PGL_n\CC$. Then, thinking of the Lie algebra $\mathfrak g$ as trace-less $n\times n$ matrices, we may choose the Cartan sub-algebra $\fraka$ to be the trace-less (real) diagonal matrices and so the algebra $\frakb$ from the previous section is a sub-space of trace-less diagonal matrices.
These coordinates are particular nice for calculating limit groups.

Let $X \in \frakb$, and let  $E_0, \ldots, E_k$ be the eigenspaces of $X$, listed in order of increasing eigenvalue. Consider the partial flag $\flag = \flag (X)$, defined to be the chain of subspaces $V_0 \supset V_1 \supset \ldots \supset V_k$,
where $$V_j = E_j  \oplus \cdots \oplus E_k$$
and by convention, we take $V_{k+1} = \{0\}$.  We define the \emph{partial flag group} $\PGL(\flag)$ to be the subgroup of $\PGL_n \RR$ which stabilizes $\flag$.  There is a natural surjection
$$\pi_{\flag}:\PGL(\flag) \longrightarrow \operatorname{P} \big(\bigoplus_{i=0}^k \GL(V_i/V_{i+1})\big),$$
where on the right-hand side, $\operatorname{P}$ denotes the projection $\GL_n \to \PGL_n$.
We call the group $U({\flag})=\ker(\pi_{\flag})$ the {\em flag unipotent subgroup}.  It is connected
and unipotent, however in general it is not maximal among unipotent subgroups of $\PGL({\flag})$. Here is a simple corollary of Theorem~\ref{thm:symmetric-limits}:

\begin{corollary}
\label{thm:symmetric-PGL}
Let $L = Z_{H}(X) \ltimes N_+(X)$ be any limit of ${H} \subset \PGL_n \RR $ as in Theorem~\ref{thm:symmetric-limits}, where $X \in \frakb$. Then $L \subset \PGL(\flag)$. Further, the group $Z_{H}(X)$ consists of all elements of $H$ which preserve the decomposition $\RR^n = E_0 \oplus \cdots \oplus E_k$, while $N_+(X) = U_\flag$. So, in a basis respecting the decomposition $\RR^n = E_0 \oplus \cdots \oplus E_k$, every element of $L$ has the form:
 $$\left(\begin{array}{ccccc}
A_0 & 0 & 0 &\cdots & 0\\
* & A_1 & 0 &\cdots & 0\\
* & * & A_2&\cdots& 0\\
\cdots &\cdots &\cdots &\cdots &\cdots \\
* & * & * &\cdots & A_k
\end{array}\right)$$
where $\operatorname{diag}(A_0,\ldots, A_k) \in H \cap \operatorname{P}(\GL(E_0)\oplus \cdots \oplus \GL(E_k))$ and each $*$ denotes a block which may take arbitrary values.
\end{corollary}

\begin{proof}
That $Z_H(X)$ consists of all elements of $H$ that preserve the eigenspaces of $X$ is clear. Next, we examine the action by conjugation of $\exp(t X)$ on $\PGL_n$. Writing an element $g \in \PGL_n \KK$ in block form with respect to the decomposition $\KK^n = E_0 \oplus \cdots \oplus E_k$, we see that conjugation by $\exp(-t X)$ multiplies the $(i,j)$ block by the scalar $e^{-t(d_i - d_j)}$, where $d_i$ is the $i^{th}$ eigenvalue of $X$ with eigenspace $E_i$. Therefore, $\exp(-tX)g\exp(tX) \to 1$ as $t \to \infty$ if and only if $g \in \PGL(\flag)$ and $\pi_\flag(g) = 1$.
\end{proof}

\subsection{Symmetric subgroups of $G = \GL_n \RR$}
\label{sec:symmetric-GLn}
Although the general linear group $\GL_n \RR$ is not semi-simple, we abuse terminology and 
call a subgroup $H < \GL_n \RR$ symmetric if there exists an involution $\sigma$, commuting with the Cartan involution $\theta$, such that $H$ is the set of fixed points of $\sigma$.
 It is sometimes more convenient to work with symmetric subgroups in $\GL_n \RR$ than with symmetric subgroups in $\PGL_n \RR$ or $\SL_n \RR$.
 Theorems~\ref{thm:symmetric-limits} and~\ref{thm:limit_nice} do not directly apply in this setting. However, one may determine the limits of a symmetric subgroup $H$ of $\GL_n \RR$ by applying the theorems to either the image $\operatorname{P}H$ under the projection $\operatorname{P}: \GL_n \to \PGL_n$, or to $H \cap \SL_n \RR$. This strategy will be employed in the following sections.

We now list the symmetric subgroups of $\GL_n \RR$. 
Any involution of $\GL_n \RR$ commutes with a Cartan involution $\theta$. So, we take $\theta$ to be the standard Cartan involution, given by $\theta(X) = X^{-t}$ and give a list of all involutions that commute with~$\theta$. 

\begin{itemize}
\item[(1a)] First, there are inner involutions of the form $\sigma(X) = J X J^{-1}$, where $J = -I_p \oplus I_q$ for some $p +q = n$. Note that $J^2 = Id$ and $J \in K$, so $\sigma$ commutes with $\theta$. In this case the symmetric subgroup of fixed points of $\sigma$ is:
$$H_\sigma = Z_{\GL_n\RR}(J) = \GL_p \RR\oplus \GL_q\RR.$$
\item[(1b)]  For $n=2m$ even, let $J$
be the complex structure on ${\mathbb R}^{2m}$ given by $m$ copies of the standard complex structure on $\RR^2$:
$$J=\bigoplus_m \left[\begin{array}{cc}0 & -1 \\1 & 0\end{array}\right].$$
Again, $J$ is orthogonal, so the involution $\sigma$ defined by $\sigma(X) = JXJ^{-1}$ commutes with~$\theta$. Then,
$$H_{\sigma}= Z_{\GL_n\RR}(J) =: \GL_m{\mathbb C}.$$
\item[(2)] Of course there is the Cartan involution $\theta$ itself. In this case $H_\theta = \OO(n)$.
\item[(3)] Let $\phi$ be the involution defined by $\phi(X)=|\det X|^{-2/n}X$. Then $H_\phi =\SL^\pm_n \mathbb R := \{ A \in \GL_n \RR: \det A = \pm 1\}.$
\item[(4)] Assume $n= 2m$ is even. Let $\varphi$ be the involution that is the identity on matrices of positive determinant and multiplication by $-1$ on matrices of negative determinant.
Then $H_{\varphi} = \GL^+_n \RR =: \{A \in \GL_n \RR : \det A > 0\}$ is the identity component of $\GL_n \RR$.
\item[(5)] In fact, $\theta, \phi$ and $\varphi$ commute with each other and with any inner involution $\sigma$ of type (1a) or (1b). So if $\epsilon_1, \epsilon_2, \epsilon_3 \in \{0,1\}$ and $\sigma$ is any inner involution of type (1a) or (1b), then $\tau = \sigma \theta^{\epsilon_1}\phi^{\epsilon_2}\varphi^{\epsilon_3}$ defines an involution.
\end{itemize}

Since the involutions of types (3) and (4) are not so interesting, we will be most interested in the inner involutions $\sigma$ of type (1a) and (1b) and their products $\tau = \sigma \theta$ with the Cartan involution. 
If $\sigma$ is an inner involution of type (1a), then $H_\tau$ is the orthogonal group: $$H_\tau = \OO(p,q).$$
 If $\sigma$ is an inner involution of type (1b), then $H_\tau$ is the symplectic group:
$$H_\tau = \Sp(2m,\RR).$$

\begin{proposition} Every continuous involution of $\GL_n \mathbb R$ is conjugate to one listed above.
\end{proposition}
\begin{proof}  
The outer (continuous) automorphism group of $\SL_n{\mathbb R}$ is cyclic
of order $2$ generated by the Cartan involution $\theta$ for $n\ge3$ and trivial otherwise (see Theorem~4.5 of \cite{delp}).
Since $\GL_n \RR\cong\SL_n^{\pm}\times\RR$, it follows that the outer automorphisms of the identity component $\GL^+_n \RR$ are 
$\operatorname{Out}(\GL^+_n \RR)\cong{\mathbb Z}_2\times \operatorname{Aut}({\mathbb R})$, 
generated by  $\theta$ and $\phi$. If $n$ is even, then there is exactly one nontrivial outer involution of $\GL_n \RR$ which fixes the identity component, namely $\varphi$. If $n$ is odd, there are no such outer involutions.

Every inner involution is given by conjugation
by some $J$ with $J^2$ central. 
We may assume $\det J=\pm1$. Then $J$ is conjugate to one of the matrices listed.
\end{proof}

\noindent We next apply Theorem~\ref{thm:limit_nice} to determine the limits in the most interesting cases.

\subsection{Limits of $\GL_m{\mathbb C}$ in $\GL_{2m}{\mathbb R}$}

Consider the involution $\sigma$ of $\GL_{2m} \RR$ defined by $\sigma(X) = JXJ^{-1}$ where $$J = \bpmat 0 & I_m \\ -I_m & 0 \epmat.$$
Then the fixed point set of $\sigma$ naturally identifies with the group $\GL_m \CC$.
The standard basis $e_1, \ldots, e_m, e_{m+1}, \ldots, e_{2m}$ is compatible with the complex structure $J$ in the sense that $e_{j+m} = Je_j$ and $e_j = -J e_{j+m}$. 
Note that the center $Z(\GL_{2m} \RR)$ is contained in $\GL_m \CC$. Therefore, to determine the limits of $\GL_m \CC$ inside of $\GL_{2m} \RR$, we pass to the projective general linear group via the quotient map $\operatorname{P}: \GL_{2m}\RR \to \PGL_{2m} \RR$. The image of $\GL_m \CC$ is the subgroup $H = \operatorname{P}( \GL_m \CC)  \cong \GL_m \CC / \RR^*$. Note that $\sigma$ is well-defined on $\PGL_{2m} \RR$ and that $H$ is precisely the fixed point set of $\sigma$. So we may apply Theorem~\ref{thm:limit_nice} to determine the limits of $H$ inside of $G = \PGL_{2m} \RR$.
All limits of $\GL_m \CC$ in $\GL_{2m} \RR$ are of the form $\operatorname{P}^{-1} L$ where $L \subset \PGL_{2m} \RR$ is a limit of $H$.

Note that $\sigma$ commutes with the standard Cartan involution $\theta$. The $-1$ eigenspace $\frakq$ of $\sigma$ is given by matrices of the form
$$ \bpmat A & B\\ B & -A \epmat $$
A  maximal abelian subgroup $\frakb$ of $\frakp \cap \frakq$ is given by elements $X$ of the form
$$X = \bpmat D & 0 \\ 0 & -D \epmat$$
where $D = \operatorname{diag}(d_1,\ldots,d_m)$ is a diagonal $m \times m$ matrix.  The system of positive simple roots of $\frakg\frakl_{2m} \RR$ with respect to $\frak b$ can be chosen to be $$\Delta = \left\{ d_{i+1} - d_{i} \right\}_{i=1}^{m-1} \cup \{ 2 d_m \}.$$
In this case the inclusion $W_{H \cap K} \hookrightarrow W$ is an isomorphism; both Weyl groups simply permute the diagonal entries of $D$ and also the signs.
Therefore, $\mathcal W = \{1\}$ and we may take $\overline{\frakb^+}$ to be the collection of diagonal matrices $X$ as above, where $0 \leq d_1 \leq \ldots \leq d_m $.
Then, by Theorem~\ref{thm:limit_nice}, the conjugacy classes of limits of $H$ in $G$ are enumerated by subsets $I \subset \Delta$. For a given subset $I \subset \Delta$, the corresponding limit group $$L_I = Z_H(\frakb_I) \ltimes N_I = Z_H(X) \ltimes N_+(X)$$ is the limit under conjugacy by $\exp(t X)$ as $t \to \infty$, where $X \in \frakb_I^{+}$. 
Let $$-\lambda_k < -\lambda_{k-1} < \cdots < \lambda_0 = 0 < \lambda_1 < \cdots < \lambda_k$$
denote the eigenvalues of $X$ (symmetric under negation). Note that either $\lambda_j$ or $- \lambda_j$ is a diagonal entry of $D$.
Note also that the eigenspaces $E_{\lambda_j}$ satisfy $E_{\lambda_j} = J E_{-\lambda_j}$, and in particular the zero eigenspace $E_0$ is a complex subspace, invariant under $J$.
Hence any element $g \in Z_H(X)$ preserves the eigenspace decomposition $\RR^{2m} = E_0 \bigoplus_{j=1}^k E_{\lambda_j} \oplus E_{-\lambda_j}$. The action of $g$ on $E_0$ is $J$-linear. The action of $g$ on $E_{\lambda_j} \oplus E_{-\lambda_j}$ is also $J$-linear and further preserves the real and imaginary parts $E_{\lambda_j}$ and $E_{-\lambda_j} = J E_{\lambda_j}$. Therefore, the matrix for the action of $g$ on $E_{\lambda_j} \oplus E_{-\lambda_j}$ has the form $$\bpmat A_j & 0 \\ 0 & A_j \epmat$$ in a basis which is the union of a basis for $E_{\lambda_j}$ and $J$ times that basis (which is a basis for $E_{-\lambda_j}$). This characterizes $Z_H(X)$.
Next, the flag $\flag$ defining the unipotent part $U(\flag) = N_I = N_+(X)$ of $L$ (see Section~\ref{sec:symmetric-PGL}) is given by the sub-spaces $$V_{-k} \supset V_{-(k-1)} \supset \cdots V_0 \supset V_{1} \supset \cdots \supset V_{k}$$
where $V_j = E_{\lambda_j} \oplus \cdots \oplus E_{\lambda_{k}},$ where $\lambda_{-j} := -\lambda_j$. 

We may explicitly describe the corresponding limit $\operatorname{P}^{-1} L$ of $\GL_m \CC$ in $\GL_{2m} \RR$.
In a basis respecting the ordered eigen-space decomposition $$\RR^{2m} := E_{\lambda_k} \oplus \cdots \oplus E_0 \oplus \cdots E_{-\lambda_k}$$
and such that the basis elements for $E_{-\lambda_j}$ are $J$ times the basis elements for $E_{\lambda_j}$ (where $\lambda_j > 0$), the elements of the limit group $\operatorname{P}^{-1} L$ are exactly the matrices of the form:

$$\bpmat A_k & & & & & & & &\\  * & A_{k-1} & & & & & & & \\ \vdots & \vdots & \ddots & & & & & & \\ * & * & \cdots & A_{1} & & & & & \\ * & * & \cdots & * & \bmat A_0 & B_0\\ -B_0 & A_0 \emat & & &\\
 * & * & \cdots & * & * & A_{1} & &\\ \vdots & \vdots & \cdots & \vdots & \vdots & \vdots & \ddots &\\ * & * & \cdots & * & * & * & \cdots & A_{k-1}\\ * & * & \cdots & * & * & * & \cdots & * & A_k \epmat$$

 \noindent where each $A_j$ is a square matrix with dimension $\operatorname{dim}(E_j)$, which is  equal to the multiplicity of the eigenvalue $\lambda_j$ of $X$. Each $*$ is an arbitrary block matrix of the appropriate dimensions, and the (blank) upper diagonal entries are all zero. The central block $\bpmat A_0 & B_0\\ -B_0 & A_0 \epmat$ describes the $J$-linear transformations of the $0$-eigenspace $E_0$ of $X$; this block only appears if the root $2 d_m \in I$.
  The group $\operatorname{P}^{-1}( Z_H(\frakb_I))$ is the subgroup for which all $*$ blocks are zero. The group $\operatorname{P}^{-1} N_I \cong N_I$ is the subgroup for which all diagonal blocks $A_j$ are the identity (and $B_0 = 0$).

\begin{remark}
In the case that $D = \operatorname{diag}(1,1,\ldots,1)$, we have only two eigenspaces $E_1$ and $JE_1$. The limit group is then the centralizer of the matrix $\bpmat 0 & 0 \\ I_m & 0 \epmat$ which we should think of as a degeneration of the complex structure. The limit group is isomorphic to the general linear group in dimension $m$ over the ring $\RR[\epsilon]/(\epsilon^2)$.
\end{remark}


\subsection{Limits of $\Sp(2m)$ in $\GL_{2m} \RR$}

Consider $H = \Sp(2m, \RR)$ inside of $\GL_{2m} \RR$. Note that $H \subset \SL_{2m} \RR$. To determine the conjugacy classes of limits of $H$ in $\GL_{2m} \RR$ it suffices to determine the limits of $H$ in $G = \SL_{2m} \RR$. We apply Theorem~\ref{thm:limit_nice}.
The defining involution for $H$, at the Lie algebra level, is $\tau(X) = \sigma(\theta(X)) = -JX^TJ^{-1}$ where $\theta$ is the standard Cartan involution, and $\sigma$ is conjugation by a complex structure $J$ fixed by $\theta$. In this case, it will be more convenient to take $$J = \bpmat J_0 &0 &\cdots & 0 \\ 0 & J_0 & & 0\\ 0 & 0 & \ddots & & \\ 0 & 0 & \cdots & J_0 \epmat,$$ where $J_0 = \bpmat 0 & 1\\ -1 & 0\epmat$.  The $-1$ eigenspace $\frakq$ of $\tau$ is given by matrices of the form $(A_{jk})_{j,k=1}^m$ where each $A_{jk}$ is a $2 \times 2$ block with 
$A_{jk} = -J_0 A_{kj}^T J_0$. The diagonal blocks of elements of $\frakq$ have the form $A_{jj} = \bpmat d_j & 0\\ 0 & d_j \epmat$.
A maximal abelian sub-algebra $\frakb$ of $\frakp \cap \frakq$ is given by matrices of the form
$$D = \bpmat D_1 &0 &\cdots & 0 \\ 0 & D_2 & & 0\\ 0 & 0 & \ddots & & \\ 0 & 0 & \cdots & D_m \epmat,$$ where $D_j =  \bpmat d_j & 0\\ 0 & d_j \epmat$ and $d_1 + \cdots + d_m = 0$.  The system of positive simple roots can be chosen to be 
$$\Delta = \left\{ d_{i+1} - d_{i} \right\}_{i=1}^{m-1}.$$
In this case the inclusion $W_{H \cap K} \hookrightarrow W$ is an isomorphism; both Weyl groups simply permute the block diagonal entries of $D$.
Therefore, $\mathcal W = \{1\}$ and $\overline{\frakb^+}$ is the collection of diagonal matrices $X$ as above, where $d_1 \leq \ldots \leq d_m$. 
Then, by Theorem~\ref{thm:limit_nice}, the conjugacy classes of limits of $H$ in $G$ are enumerated by subsets $I \subset \Delta$. For a given subset $I \subset \Delta$, the corresponding limit group $$L_I = Z_H(\frakb_I) \ltimes N_I = Z_H(X) \ltimes N_+(X)$$ is the limit under conjugacy by $\exp(t X)$ as $t \to \infty$, where $X \in \frakb_I^{+}$.  Here, $Z_H(\frakb_I) = Z_H(X)$ has block form:
$$Z_H(\frakb_I) = \bpmat \Sp(2 m_1) & 0 & \cdots & 0 \\ 0 & \Sp(2 m_2) & \cdots & 0\\ \vdots & \vdots & \ddots & \vdots \\ 0 & 0 & \cdots & \Sp(2m_k) \epmat,$$
where each block $\Sp(2m_j)$ is the symplectic group of dimension $2 m_j$ consisting of those elements of $\Sp(2m)$ which preserve the $j^{th}$ eigenspace of $X$ and act as the identity on the other eigenspaces of~$X$. The flag $\flag$ preserved by $L$  is given by $V_0 \supset \ldots \supset V_k$, where $V_j = E_j \oplus \cdots \oplus E_k$ is the direct sum of the last $k-j+1$ eigenspaces $E_i$ of $X$, where the $E_i$ are indexed in order of increasing eigenvalue. The flag unipotent subgroup $N_I = U(\flag)$ (see Section \ref{sec:symmetric-PGL}) has block structure:
$$N_I = \bpmat I_{m_1}& 0 & \cdots & 0 \\ * & I_{m_2} & \cdots & 0\\ \vdots & \vdots & \ddots & \vdots \\ * & * & \cdots & I_{m_k} \epmat,$$
where all lower diagonal blocks are labeled $*$ to denote that the entries are arbitrary.

\subsection{Limits of $\GL_p \oplus \GL_q$ in $\GL_{p+q} \RR$}

Consider the involution $\sigma$ defined by $\sigma(X) = JXJ^{-1}$ where $J = \bpmat -I_p & 0 \\ 0 & I_q \epmat.$
We assume without loss of generality that $q \geq p$ and set $r = q -p$.
The fixed set of $\sigma$ naturally identifies with $\GL_p \RR \oplus \GL_q \RR$.
Note that the center $Z(\GL_{p+q})$ is contained in $\GL_p \oplus \GL_q$. Therefore the involution $\sigma$ is well-defined on the quotient $\PGL_{p+q} \RR$. The image $H = \operatorname{P}(\GL_p \oplus \GL_q)$ under the projection $\operatorname{P}: \GL_{p+q} \to \PGL_{p+q}$ is therefore a symmetric subgroup of $\PGL_{p+q} \RR$ and we may apply Theorem~\ref{thm:limit_nice} to determine the limit groups $L$ of $H$. Then any limit of $\GL_p \RR \oplus \GL_q \RR$ in $\GL_{p+q} \RR$ is conjugate to $\operatorname{P}^{-1} L$ where $L$ is some limit of $H$.

Note that $\sigma$ commutes with the standard Cartan involution~$\theta$. 
The $-1$ eigenspace $\frakq$ of $\sigma$ is given by matrices of the form
$$\bpmat 0_{p \times p} & B \\ C & 0_{q \times q} \epmat$$
where $B$ is $p \times q$ and $C$ is $q \times p$.
A maximal abelian sub-algebra $\frakb$ of $\frakp \cap \frakq$ is given by matrices of the form
$$X = \bpmat 0_{p \times p} & D & \\ D & 0_{p\times p} &  \\  & & 0_{r \times r} \epmat$$
where $D = \operatorname{diag}(d_1, \ldots, d_p)$ is a $p \times p$ diagonal matrix. The system of positive simple roots can be chosen to be $$\Delta = \left\{ d_{i+1} - d_{i} \right\}_{i=1}^{p-1} \cup \{ 2 d_p \}.$$
In this case the inclusion $W_{H \cap K} \hookrightarrow W$ is an isomorphism; both Weyl groups simply permute the diagonal entries of $D$ and also the signs.
Therefore, $\mathcal W = \{1\}$ and $\overline{\frakb^{+}}$ is the collection of matrices $X$ as above, where $0 \leq d_1 \leq \ldots \leq d_p$.
Then, by Theorem~\ref{thm:limit_nice}, the conjugacy classes of limits of $H$ in $G$ are enumerated by subsets $I \subset \Delta$. For a given subset $I \subset \Delta$, the corresponding limit group $$L_I = Z_H(\frakb_I) \ltimes N_I = Z_H(X) \ltimes N_+(X)$$ is the limit under conjugacy by $\exp(t X)$ as $t \to \infty$, where $X \in \frakb_I^{+}$.

Let's see more explicitly what this group $L_I$ looks like. The eigenvalues of $X$ are 
$$-\lambda_k < -\lambda_{k-1} < \cdots < \lambda_0 = 0 < \lambda_1 < \cdots < \lambda_k$$
 where if $j \neq 0$,  $\lambda_j = 2 d_{i_j}$ is twice one of the diagonal elements of $D$. The multiplicity $m_j$ of $\lambda_j$ is determined by the subset $I$. The eigenvalue $\lambda_0 = 0$ has multiplicity equal to $r + 2m_0$ where $m_0$ is the number of the $d_i$ which are zero. 
Note also that $E_{\lambda_j} = J E_{-\lambda_j}$. Hence an element $h \in Z_H(\frakb_I)$ preserves both $E_{\lambda_j}$ and $E_{-\lambda_j}$ and has identical matrix on both subspaces, when the basis for $E_{-\lambda_j}$ is taken to be $J$ times the basis for $E_{\lambda_j}$; we work in such a basis. Also, the zero eigenspace $E_0$ is invariant under $J$; the elements of $H$ which preserve $E_0$ form a copy of $\GL(r+m_0)\oplus \GL(m_0)$.

Next, the flag $\flag$ defining the unipotent part $U(\flag) = N_I = N_+(X)$ of $L$ (see Section~\ref{sec:symmetric-PGL}) is given by the sub-spaces $V_{-k} \supset V_{-(k-1)} \supset \cdots V_0 \supset V_{1} \supset \cdots \supset V_{k}$
where $V_j = E_{\lambda_j} \oplus \cdots \oplus E_{\lambda_{k}},$ and where $\lambda_{-j} := -\lambda_j$. Therefore $N_I = U(\flag)$ is the unipotent group which is block lower diagonal in a basis respecting the ordered decomposition $\RR^n = E_{-\lambda_k} \oplus \cdots \oplus E_0 \oplus \cdots \oplus E_{\lambda_k}$ into eigenspaces.
Therefore, in such a basis, the elements of the corresponding limit $\operatorname{P}^{-1} L_I$ of $\GL_p \oplus \GL_q$ in $\GL_{p+q}$ have matrix form

$$\bpmat A_k & & & & & & & &\\ * & A_{k-1} & & & & & & & \\ \vdots & \vdots & \ddots & & & & & & \\ * & * & \cdots & A_{1} & & & & & \\ * & * & \cdots & * & \bmat A_0 & 0\\ 0 & B_0 \emat & & &\\
 * & * & \cdots & * & * & A_{1} & &\\ \vdots & \vdots & \cdots & \vdots & \vdots & \vdots & \ddots &\\ * & * & \cdots & * & * & * & \cdots & A_{k-1}\\ * & * & \cdots & * & * & * & \cdots & * & A_k \epmat$$
  
\noindent where for $j \neq 0$, the matrix $A_j$ is the square $m_j \times m_j$ matrix representing both the action on $E_{\lambda_j}$ and on $E_{-\lambda_j} = J E_{\lambda_j}$, and the block matrix $\bpmat A_0 & 0 \\ 0 & B_0 \epmat$ represents an element of $\GL(r+m_0)\oplus \GL(m_0)$ corresponding to the action on $E_0$. As always, the $*$ blocks are arbitrary.

\subsection{Limits of $\OO(p,q)$ in $\GL_{p+q}\RR$}
\label{sec:Opq}

Set $p+q = n$. Let $\tau$ be the involution of $\GL_{p+q}$ defined by $\tau(g) = Jg^{-T}J^{-1}$ where $J$ is the matrix
$$J = \bpmat -I_p & \\ & I_q\epmat.$$
The group $\OO(p,q)$ is the fixed point set of $\tau$. Consider the image $\PO(p,q)$ of $\OO(p,q)$ under the projection $\operatorname{P}: \GL_{n} \RR \to \PGL_{n} \RR$.  To determine the limits of $\OO(p,q)$ in $\GL_{n} \RR$, it suffices to determine the limits of $H = \PO(p,q)$ in $G = \PGL_{n} \RR$. For, $\OO(p,q)$ is the intersection of $\operatorname{P}^{-1}( \PO(p,q))$ with the subgroup $\SL^\pm$ of matrices of determinant $\pm 1$. Therefore, all limits of $\OO(p,q)$ are of the form $\operatorname{P}^{-1}(L) \cap \SL^\pm$ where $L$ is a limit of $H$ in $G$. Note also that $\tau$ preserves the center $Z(\GL_{n})$, therefore $\tau$ descends to a well-defined involution of the projective general linear group; its fixed point set is exactly $H = \PO(p,q)$. So we may apply Theorem~\ref{thm:limit_nice} to determine the limits of $H$.

At the Lie algebra level, our involution has the form $\tau(X) =  -JX^TJ^{-1}.$
A maximal abelian sub-algebra $\frakb$ of $\frakp \cap \frakq$ is in this case given by the full Cartan sub-algebra $\fraka$ of traceless diagonal matrices 
$$D = \bpmat d_1 & & & \\ & d_2 & & \\ & & \ddots & \\ & & & d_n \epmat.$$ The  system of positive simple roots can be chosen to be  $$\Delta = \left\{ d_{i+1} - d_{i} \right\}_{i=1}^{n-1}.$$
In this case, the Weyl group $W$ is the full symmetric group $S_{p+q}$ permuting the standard basis of $\RR^{p+q}$. The closed Weyl chamber $\overline{\frakb^+}$ corresponding to $\Delta$ consists of the diagonal matrices $X$, as above, such that $d_1 \leq d_2 \leq \cdots \leq d_{n}$.
The Weyl group $W_{H \cap K}$ for $\Sigma(\frakg_{\tau \theta}, \frakb)$ is given by the permutations $S_p \times S_q$ of the standard basis which preserve the signature; in other words $W_{H \cap K}$ permutes the first $p$ basis vectors and the last $q$ coordinate directions independently. 
A closed Weyl chamber $\mathcal W \cdot \overline{\frakb^+}$ for $\Sigma(\frakg_{\tau \theta}, \frakb)$ is given by the diagonal matrices $\operatorname{diag}(d_1, \ldots, d_n)$ such that $d_i \leq d_{i+1}$ if $i = 1,\ldots, p-1$ or if $i = p+1,\ldots, p+q = n$. Then $\mathcal W$ consists of permutations $\varpi$ of the following form. For some $1 \leq k \leq p$ (assuming $p \leq q$), and two sets of $k$ indices $1 \leq i_1 < \cdots < i_k \leq p$ and $p+1 \leq j_1 < \cdots < j_k \leq p+q = n$, we have that $\varpi(i_r) = p+r$ and $\varpi(j_r) = p -k +r$ and the remaining $p-k$ indices between $1$ and $p$ are mapped in order to the first (smallest) $p-k$ indices, while the remaining $q -k$ indices between $p+1$ and $n$ are mapped in order to last (largest) $q-k$ indices. In fact, this specific form of $\mathcal W$ is not important; we may work with any collection $\mathcal W$ of coset representatives of $W/W_{H \cap K}$.

By Theorem~\ref{thm:limit_nice}, the conjugacy classes of limits of $H$ in $G$ are enumerated (with redundancy) by subsets $I \subset \Delta$ and elements $w \in \mathcal W$. For a given subset $I \subset \Delta$ and $w \in \mathcal W$, the corresponding limit group $$L_{I ,w}= Z_{H_w}(\frakb_I) \ltimes N_I = Z_{H_w}(X) \times N_+(X)$$ is the limit of $H_w = wHw^{-1}$ under conjugacy by $\exp(t X)$ as $t \to \infty$, where $X \in \frakb_I^{+}$. 
Let $E_0, E_1, \ldots, E_k$ be the eigenspaces of $X$ listed in order of increasing eigenvalue. In our chosen coordinates, each $E_i$ is a span of consecutive coordinate directions. Now, $H_w$ is the fixed point set of the involution $\tau_w$ defined by $\tau_w(g) = J_w g^{-T} J_w^{-1}$, where $J_w = w J w^{-1}$ is a diagonal form of signature $(p,q)$ but with the $p$ $(-1)$'s and $q$ $(+1)$'s arranged in a (possibly) different order. Let $(p_i, q_i)$ denote the signature of $J_w$ when restricted to $E_i$. Then $p_1 + \cdots  + p_k = p$ and $q_1 + \cdots + q_k = q$, and  $Z_{H_w}(\frakb_I)$ is seen to have the block diagonal form $$Z_{H_w}(\frakb_I) = \operatorname{P} \bpmat \OO(p_1,q_1) & & & \\ & \OO(p_2,q_2) & & \\ & & \ \ \ddots \ \ & \\ & & & \OO(p_k,q_k) \epmat.$$
The full limit group $L_{I,w} \subset \PGL(\flag)$ preserves the flag $\flag$ consisting of subspaces $V_0 \supset \cdots \supset V_k$, where $V_j = E_j \oplus \cdots \oplus E_k$. The unipotent part $N_I = U(\flag)$ has the form
$$N_I = \operatorname{P} \bpmat I_{p_1+q_1} & & & \\ * & I_{p_2+q_2} & & \\ \vdots & \vdots & \ddots & \\ * & * & \cdots & I_{p_k+q_k} \epmat$$
where the upper diagonal blocks, denoted by $*$, are arbitrary.

\begin{remark}
\label{rem:SOpq-redundancy}
In this example, we may see explicitly that the finer enumeration of limit groups described by Remark~\ref{rem:finer-list} may still have redundancy.
For consider the case $p=2$, $q=2$, and consider $I = \{ d_2 - d_1, d_4 - d_3\}$. Then $W_I$ consists of four permutations, namely the group generated by transposing the first and second basis vector and transposing the third and fourth basis vector; so that $W_I = W_{H \cap K}$. One easily computes that $| W_I \backslash W / W_{H \cap K}| = 3$; representatives are given by the permutation $w_1$ that transposes the second and third basis vector, a permutation $w_2$ that exchanges the first and second basis vectors with the third and fourth, and the identity permutation $e$. However, there are only two conjugacy classes of limit group, because $L_{I,w_2} = L_{I,e}$.
\end{remark}

\subsubsection{Partial flag of quadratic forms}
Finally, we note that if $L_{I,w}$ is as above, then the corresponding limit groups $\operatorname{P}^{-1}(L_{I,w}) \cap \SL^\pm$ of $\OO(p,q)$ in $\GL_n \RR$ are easily described. We will investigate these limit groups and their corresponding geometries in depth in the next section. Here we introduce some notation to give an invariant description of the limit groups.
Let $\flag$ be the partial flag formed the chain of subspaces $V_0 \supset V_1 \supset \ldots \supset V_k$. 
A {\em partial flag of quadratic forms} $\pfbeta=(\beta_0,\cdots,\beta_k)$ on the partial flag $\flag$ is a collection of quadratic forms $\beta_i$ defined on each quotient
$V_i/V_{i+1}$ of the partial flag $\flag$. We denote the linear transformations which preserve $\flag$ and induce an isometry of each $\beta_i$ by $\Isom(\pfbeta,{\flag})$.
The {\em signature} of a non-degenerate quadratic form $\beta$ is $\eps(\beta)=(n_-,n_+)$, where $n_-$ (resp. $n_+$) is the dimension of the largest subspace on which $\beta$ is negative (resp. positive) definite.
Two quadratic forms have the same isometry group iff they
are scalar multiples of each other thus $\OO(p,q)\cong \OO(p',q')$ iff $\{p,q\}=\{p',q'\}$.
The {\em signature} of a partial flag of quadratic forms $\pfbeta=(\beta_0,\cdots,\beta_k)$ is
$$\eps(\pfbeta)=(\eps(\beta_0),\cdots,\eps(\beta_k))=((p_0,q_0)\cdots (p_k,q_k)).$$ 
The signature $\eps(\pfbeta)$ determines $\Isom (\pfbeta, \flag)$ up to conjugation.  When $\flag$ is adapted to the standard basis and all $\beta_i$ are diagonal, we will use the notation 
\begin{align*}
\Isom(\pfbeta,\flag) &=: \OO((p_0,q_0),\cdots,(p_k,q_k))\\ &= \bpmat \OO(p_0,q_0) & &  \\ &  \ \ \ddots \ \ & \\ & & \OO(p_k,q_k) \epmat \ltimes \bpmat I_{p_0+q_0} & & & \\ * & I_{p_1+q_1} & & \\ \vdots & \vdots & \ddots & \\ * & * & \cdots & I_{p_k+q_k} \epmat.
\end{align*}
The conjugacy class of this group is unchanged by scaling some $\beta_i$. In fact $\OO((p_0,q_0),\cdots,(p_k,q_k))$ is conjugate to $\OO((p_0',q_0'),\cdots,(p_k',q_k'))$ if and only if for all $i =0,\ldots,k$, $(p_i,q_i) = (p_i',q_i')$ or $(p_i,q_i) = (q_i',p_i')$.
 As a special case observe that when $\flag$ is a full flag, then $\Isom(\pfbeta, \flag)$ is conjugate to 
$\OO((1,0),\cdots,(1,0))$, which is the group of  lower triangular matrices with diagonal entries $\pm1$.
We will adopt the convention that the signature $(p,0)$ can be denoted by $(p)$ so this group is also
written as $\OO((1),(1),\cdots,(1))$. This is in agreement with denoting $\OO(n,0)$ by $\OO(n)$. 
The application of Theorem~\ref{thm:limit_nice} above shows:

\begin{theorem}
\label{thm:limits-Opq}
The limits of $\OO(p,q)$ (resp. $\PO(p,q)$) inside of $\GL_{p+q} \RR$ are all of the form $\Isom(\pfbeta, \flag)$ (resp. $\PP \Isom(\pfbeta, \flag)$). Further $\Isom(\flag, \pfbeta)$ (resp. $\PP \Isom(\pfbeta, \flag)$) is a limit of $\OO(p,q)$ (resp. $\PO(p,q)$) if and only if the signature $((p_0,q_0),\ldots,(p_k,q_k))$ of $\pfbeta$ satisfies 
\begin{equation*} p_0 + \cdots + p_k = p \ \text{ and } \ q_0 + \cdots + q_k = q,
\end{equation*}  
after exchanging $(p_i , q_i)$ with $(q_i,p_i)$ for some collection of indices~$i$ in $\{0,\ldots,k\}$.
\end{theorem}

The groups $\PP \Isom(\pfbeta, \flag)$ are the structure groups for many interesting geometries, to be described in Section~\ref{sec:pfqf}. As a corollary to Theorem~\ref{thm:limits-Opq}, we characterize all limits of these groups.

\begin{corollary}\label{cor:limits-pfqf}
Every conjugacy limit of $\Isom(\flag, \pfbeta)$ (resp. $\PP \Isom(\flag, \pfbeta)$) is of the form $\Isom(\flag', \pfbeta')$ (resp. $\PP \Isom(\flag', \pfbeta')$). Further, up to conjugation, the flag $\flag' = \{V'_j\}$ is a refinement of $\flag = \{ V_i\}$ and the signature $\eps(\pfbeta')$ is a refinement of the signature $\eps(\pfbeta)$ in the sense that
\begin{align}
\label{eqn:signature-condition}
p_i &= \sum_{j: V_i \supset V_j' \supsetneq V_{i+1}} p_j' & q_i = \sum_{j: V_i \supset V_j' \supsetneq V_{i+1}} q_j'
\end{align}
after exchanging $(p_j', q_j')$ with $(q_j', p_j')$ for some collection of indices~$j$.
Any $\Isom(\flag', \pfbeta')$ as above is realized as a limit of $\Isom(\flag, \pfbeta)$ under some sequence of conjugacies.
\end{corollary}

\begin{proof}
Let $H = \Isom(\flag, \pfbeta)$.
Consider a conjugacy limit $L = \lim_{n \to \infty} c_n H c_n^{-1}$. The space of flags having the same type as $\flag$ is compact. Thus, we may assume that $c_n \in \PGL(\flag)$ for all~$n$, and therefore that $L \subset \GL(\flag)$. Note that $U(\flag)$ is preserved by conjugation by $c_n$. Therefore $U(\flag) \subset L$. It remains to determine the projections $\pi_i(L)$ where $\pi_i: \GL(\flag) \to \GL(V_i/V_{i+1})$ is the natural projection map. Now, $\pi_i(H) = \Isom(\beta_i) \cong \OO(p_i,q_i)$. The projection $\pi_i(L)$ is the limit of $\pi_i(H)$ under conjugation by the projections $\pi_i(c_n)$. Hence, Theorem~\ref{thm:limits-Opq} implies that $\pi_i(L) = \Isom(\flag^{(i)}, \pfbeta^{(i)})$, where $(\flag^{(i)}, \pfbeta^{(i)})$ is a partial flag of quadratic forms for $V_i/V_{i+1}$. Then, let $\flag'$ be the flag of all lifts $\pi_i^{-1}(V^{(i)}_j)$ of subspaces $V^{(i)}_j$ of each flag $\flag^{(i)}$. Let $\pfbeta'$ be the flag of quadratic forms $\pi_i^*\beta^{(i)}_j$ on those subspaces determined by pullback. Then, $L = \Isom(\flag', \pfbeta')$ is as in the statement of the Corollary.

Next, to see that  any $\Isom(\pfbeta')$ as in the Corollary is achieved as a limit, note that if the condition~(\ref{eqn:signature-condition}) is satisfied, then both $\Isom(\pfbeta)$ and $\Isom(\pfbeta')$ are limits of some $\OO(p,q)$. In fact, the elements $X,X' \in \frakb$ determining the respective limits $\Isom(\pfbeta)$ and $\Isom(\pfbeta')$ of $\OO(p,q)$ have the property that any eigenspace of $X'$ is contained in an eigenspace of $X$. Further if $E_{\lambda_1}', E_{\lambda_2}'$ are eigenspaces of $X'$ corresponding to eigenvalues $\lambda_1 < \lambda_2$, then $E_{\lambda_1}' \subset E_{\mu_1}$ and $E_{\lambda_2}' \subset E_{\mu_2}$, where $E_{\mu_1}, E_{\mu_2}$ are eigenspaces of $X$ corresponding to eigenvalues $\mu_1 \leq \mu_2$. It is then easy to see that 
$$\lim_{t \to \infty} \exp(t X') \Isom(\pfbeta) \exp(-t X') = \Isom(\pfbeta').$$
\end{proof}

\section{The geometry of a partial flag of quadratic forms}
\label{sec:pfqf}

In Section~\ref{sec:Opq}, we showed that the geometric limits of $\OO(p,q)$ inside of $\GL_{p+q} \RR$ are the groups $\Isom(\pfbeta, \flag)$ preserving a partial flag of quadratic forms. In this section we will investigate the corresponding limit geometries.

\subsection{$\XX(p,q)$ geometry and its limits}
\label{sec:Hpq}
Let $\beta$ denote a quadratic form on $\RR^n$ of signature $(p,q)$. We assume that $p > 0$. Then $\operatorname{P}\Isom(\beta) \cong \PO(p,q)$ acts transitively on the space
$$\XX(p,q) := \{ [x] \in \RP^{n-1} : \beta(x) < 0 \}$$
with stabilizer isomorphic to $\OO(p-1,q)$. Therefore $\XX(p,q)$ is a semi-Riemannian space of dimension $n-1 = p+q-1$ and signature $(p-1, q)$. 
We list some familiar cases:
\begin{itemize}
\item $(\XX(n,0), \PO(n))$ is doubly covered by spherical geometry $\mathbb S^{n-1}$. 
\item $(\XX(1,n-1), \PO(1,n-1))$ is the projective model for hyperbolic geometry $\HH^{n-1}$.
\item $(\XX(2,n-2), \PO(2,n-2))$ is the projective model for anti de Sitter (AdS) geometry $\AdS^{n-1}$.
\item $(\XX(n-1,1), \PO(n-1,1))$ is the projective model for de Sitter (dS) geometry $\dS^{n-1}$.
\end{itemize}

We now describe the possible limits of $(\XX(p,q), \PO(p,q))$ as a sub-geometry of $(\RP^{n-1}, \PGL_n\RR)$.
Consider a partial flag $\flag$ equipped with a flag of quadratic forms $\pfbeta = (\beta_0, \ldots, \beta_k)$ as in Section~\ref{sec:Opq}. Let $(p_i,q_i)$ be the signature of $\beta_i$.  Define the domain $\XX(\pfbeta) \subset \RP^{n-1}$ by $$\XX(\pfbeta) := \{ [x] \in\RP^{n-1} : \beta_0(x) < 0\}.$$
Then $\PP \Isom(\flag, \pfbeta)$ acts transitively on $\XX(\pfbeta)$. When the flag and quadratic forms are adapted the standard basis, we denote $\XX(\pfbeta)$ by 
\begin{equation}\label{PFG-notation}
\XX(\pfbeta) = \XX((p_0,q_0),\ldots,(p_k,q_k)).
\end{equation}
Note that $\XX(\pfbeta)$ is non-empty if and only if $p_0 > 0$ and that as a set, the space $\XX((p_0,q_0)\ldots (p_k,q_k))$ depends only on the first signature $(p_0,q_0)$ and the dimension $n = \sum_i (p_i + q_i)$. However, we include all $k$ signatures in the notation as a reminder of the structure determined by $\PO((p_0,q_0),\ldots,(p_k,q_k))$.

\begin{theorem}\label{thm:limits-Hpq}
The conjugacy limits of $(\XX(p,q), \PO(p,q))$ inside $(\RP^{n-1}, \PGL_n)$ are all of the form $(\XX(\pfbeta), \PP \Isom(\flag, \pfbeta))$. Further, $\XX(\pfbeta)$ is a limit of $\XX(p,q)$ if and only if $p_0 \neq 0$, and the signatures $((p_0,q_0),\ldots,(p_k,q_k))$ of $\pfbeta$ partition the signature $(p,q)$ in the sense that
\begin{equation*}
p_0 + \cdots + p_k = p \ \ \text{  and  } \ \ q_0 + \cdots + q_k = q,
\end{equation*}
after exchanging $(p_i, q_i)$ with $(q_i, p_i)$ for some collection of indices~$i$ in $\{1,\ldots,k\}$ (the first signature $(p_0,q_0)$ must \emph{not} be reversed).
\end{theorem}
\noindent More generally:
\begin{theorem}
\label{thm:limits-Xbeta}
Every conjugacy limit of $(\XX(\pfbeta), \PP\Isom(\flag, \pfbeta))$ is of the form $(\XX(\pfbeta'), \PP\Isom(\flag', \pfbeta'))$. Further, up to conjugation, the flag $\flag' = \{V'_j\}$ is a refinement of $\flag = \{ V_i\}$ and the signature $\eps(\pfbeta')$ is a refinement of the signature $\eps(\pfbeta)$ in the sense that
\begin{align*}
p_i &= \sum_{j: V_i \supset V_j' \supsetneq V_{i+1}} p_j' & q_i = \sum_{j: V_i \supset V_j' \supsetneq V_{i+1}} q_j'
\end{align*}
after exchanging $(p_j', q_j')$ with $(q_j', p_j')$  for some collection of indices~$j$ excluding $j=0$ (the first signature $(p_0',q_0')$ must not be reversed).
Any such geometry $(\XX(\pfbeta'), \PP\Isom(\flag', \pfbeta'))$ is realized as a limit of $(\XX(\pfbeta), \PP\Isom(\flag, \pfbeta))$ under some sequence of conjugacies (provided $p_0' > 0$).
\end{theorem}

\begin{proof}
Let $L = \PP\Isom(\pfbeta', \flag')$ be a limit of $\PP\Isom(\pfbeta, \flag)$ under some conjugating sequence $(c_n)$ as in Corollary~\ref{cor:limits-pfqf}. Suppose that $(\XX(\pfbeta), \PP\Isom(\pfbeta, \flag))$ limits, under conjugation by $(c_n)$ to $(Y,L)$ (in the sense of Definition~\ref{def:limit-subgeom}). 
Then $Y \subset \RP^n$ is an open orbit of $L$.
There are at most two such orbits. $\XX(\pfbeta')$ is an open orbit of $L$, non-empty if and only if $p_0' > 0$. The set of positive lines, $\XX^+$, with respect to $\beta_0'$ is the other open orbit of $L$, non-empty if and only if $q_0' > 0$. Let us now show that $Y = \XX(\pfbeta')$.

 By definition, there is some $y_\infty \in Y$ such that for all $n$ sufficiently large, $y_\infty \in c_n \XX(\pfbeta)$; in other words, $\beta_0(c_n^{-1} y_\infty) < 0$.  It is easy to see, from the proof of Theorem~\ref{thm:symmetric-limits}, that $\beta_0'$ is the limit of $\lambda_n c_n^* \beta_0$ where $\lambda_n > 0$ is a sequence of positive scalars. Therefore, 
 $$\beta_0'(y_\infty) = \lim_{n \to \infty} \lambda_n \beta_0(c_n^{-1} y_\infty) \leq 0.$$
 It follows that $Y\neq \XX^+$, so we must have $Y = \XX(\pfbeta')$ as desired. 
 
Next we show that any $\XX(\pfbeta')$ as in the Theorem is achieved. As in the proof of Corollary~\ref{cor:limits-pfqf}, we note that $\XX(\pfbeta')$ and $\XX(\pfbeta)$ are both limits of some $\XX(\beta) = \XX(p,q)$ under conjugation by the one parameter groups $\exp(t X)$ resp. $ \exp(t X')$ and that the groups satisfy 
$\PP\Isom(\pfbeta') = \lim_{t \to \infty} \exp(t X') \PP\Isom(\pfbeta) \exp(-t X')$.
Further, every eigen-space of $X'$ is contained in an eigenspace of $X$ and the eigenspace $E_0'$ corresponding to the smallest eigenvalue of $X'$ is contained in the eigenspace $E_0$ corresponding to the smallest eigenvalue of $X$. Then $\beta_0'$ agrees with $\beta_0$ on $E_0'$ and is zero on all other eigenspaces of $X'$. 
Let $y_\infty \in \XX(\pfbeta') \cap \mathbb P E_0'$. Then, since $\beta_0'(y_\infty) < 0$ we have that $\beta_0(y_\infty) < 0$ and so $y_\infty \in \XX(\pfbeta)$. Since $\exp(t X') y_\infty = y_\infty$, we have shown that the conjugate sub-geometries $(\exp(t X') \XX(\pfbeta), \exp(t X') \PP \Isom(\pfbeta) \exp(-t X'))$ limit to the sub-geometry  $(\XX(\pfbeta'), \PP\Isom(\pfbeta'))$.
\end{proof}

\subsection{The geometry of $\XX(\pfbeta)$}
Let us now describe the geometry of $(\XX(\pfbeta), \PP \Isom(\flag, \pfbeta))$. We assume that the flag $\flag$ and quadratic forms $\pfbeta$ are adapted to the standard basis so that we may use the notation~(\ref{PFG-notation}).
Let $j \in \{1,\ldots,k\}$.
The action of $\Isom(\flag, \pfbeta)$ preserves the flag of quotient spaces 
$$V_0/V_j \supset V_1/V_j \supset \cdots \supset V_{j-1}/V_j,$$ which we denote $\flag/V_j$, as well as the induced flag of quadratic forms $\beta_0, \ldots, \beta_{j-1}$, which we denote by $\pfbeta/V_j$. Then $\XX(\pfbeta/V_j)$ identifies with $\XX((p_0,q_0)\ldots(p_{j-1},q_{j-1}))$ and $\Isom(\flag, \pfbeta)$ acts on $\XX(\pfbeta/V_j)$ by transformations of $\Isom(\flag/V_j, \pfbeta/V_j) \cong \OO((p_0,q_0),\ldots,(p_{j-1},q_{j-1}))$.

There is a projection map $\pi_{j}$:
\begin{displaymath}
    \xymatrix{   V_{j}/V_{j+1} \ar[r] & \XX(\pfbeta/V_{j+1}) \ar[d]^{\pi_{j}}\\
    & \XX(\pfbeta/V_{j})}
\end{displaymath}
 which is the restriction of the natural projection map $\mathbb P(V_0/V_{j+1}) \setminus \mathbb P(V_{j}/V_{j+1}) \to \mathbb P(V_0/V_{j}).$
 The fiber $\pi_j^{-1}([x + V_{j}]) = \{[x + v + V_{j+1}] : v \in V_{j}/V_{j+1}\}$ identifies with $V_{j}/V_{j+1}$; the identification depends on the choice of representative $x$. Via the identification, $\beta_{j}$ induces an affine pseudo-metric $\varrho_{j}$ of signature $(p_{j},q_{j})$ on each fiber, defined by:
$$\varrho^2_{j}(x + v, x+w) := \beta_{j}(v-w).$$
 The metric is well-defined, provided that the representative $x + V_j$ is always chosen to satisfy $\beta_0(x+V_j) = -1$ (there are two such choices). The action of $\Isom(\flag, \pfbeta)$ on $\XX(\pfbeta/V_{j+1})$ preserves this fibration by affine spaces and preserves the affine $(p_{j},q_{j})$ metric $\varrho_j^2$ on the fibers. 
 
Alternatively, we will use the notation
\begin{displaymath}
    \xymatrix{   \mathbb A^{p_{j},q_{j}}\ar[r] & \XX((p_0,q_0),\ldots,(p_{j},q_{j}))\ar[d]^{\pi_{j}}\\
    & \XX((p_0,q_0),\ldots, (p_{j-1},q_{j-1})) }
\end{displaymath}
where the notation $\mathbb A^{p_{j},q_{j}}$ indicates that the fibers are affine spaces equipped with an affine pseudo-metric of signature $(p_{j},q_{j})$.
Combining this information for all possible values of $j$, we see that $\XX(\pfbeta)$ is equipped with an \emph{iterated affine bundle} structure (see Figure~\ref{fig:iterated-bundle}), and the fibers at each iteration are equipped with an invariant affine pseudo-metric:

\begin{equation} \label{tower}
    \xymatrix{   \mathbb A^{p_{k},q_{k}}\ar[r] & \XX((p_0,q_0),\ldots,(p_{k},q_{k}))\ar[d]^{\pi_{k}}\\
     \mathbb A^{p_{k-1},q_{k-1}}\ar[r] & \XX((p_0,q_0),\ldots, (p_{k-1},q_{k-1}))\ar[d]^{\pi_{k-1}}\\
     & \vdots\ar[d]^{\pi_2} \\
     \mathbb A^{p_{1},q_{1}}\ar[r] & \XX((p_0,q_0), (p_{1},q_{1}))\ar[d]^{\pi_1}\\
     & \XX(p_0,q_0) }
\end{equation}

We note that in the case $(p_0, q_0) = (1,0)$ the base of the tower of fibrations is a point and the next space up $\XX((1,0)(p_1,q_1))$ is an affine space $\mathbb A^{p_1,q_1}$ equipped with an invariant affine pseudo-metric of signature $(p_1,q_1)$. In the case that $(p_1,q_1)$ is $(n,0)$ or $(0,n)$, then $\mathbb A^{p_1,q_1}$ identifies with the Euclidean space $\mathbb E^n$.

\begin{remark}
In the context of Theorem~\ref{thm:limits-Hpq}, the different levels of the tower of fibrations~(\ref{tower}) correspond to different rates of collapse of $\XX(p,q)$. The initial projection $\pi_k$ should be thought of as a collapse map which collapses the directions in $\XX(p,q)$ most distorted by the conjugation action. 
\end{remark}

\begin{remark}
\label{rem:double-cover}
It is sometimes easier to work in the double cover $\widetilde \XX(\pfbeta)$ of $\XX(\pfbeta)$, which is naturally described by the hyperboloid $\beta_0 = -1$. In this case, each projection map $\pi_j$ is just the restriction of the quotient map $V_0/V_{j+1} \to V_0/V_j$ to the hyperboloid $\beta_0 = -1$ in $V_0/V_{j+1}$. It is natural to call $\widetilde \XX(\pfbeta)$ the \emph{hyperboloid model}.
\end{remark}

\begin{remark}
The action of the unipotent part $U(\flag)$ preserves each affine fiber of each fibration in \eqref{tower}. The action on each fiber is simply a translation. However, the amount and direction of translation may vary from fiber to fiber (with respect to some chosen trivialization), so that the fibers appear to shear with respect to one another.
\end{remark}

\begin{figure}
{\centering

\def\svgwidth{3.0in}
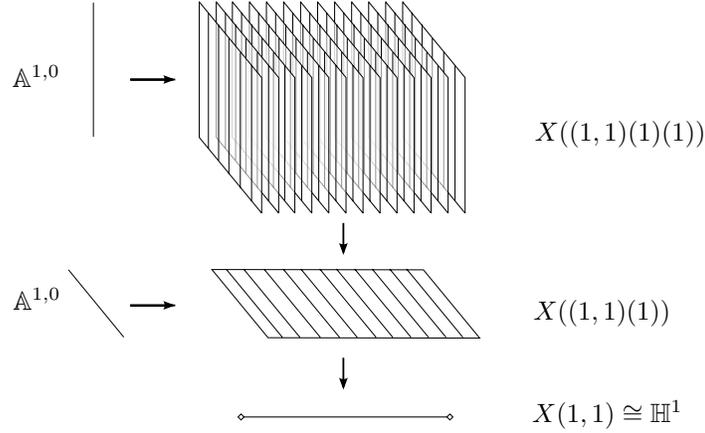

}
\caption{\label{fig:iterated-bundle} A picture of the iterated affine bundle structure of $X((1,1)(1)(1))$.}
\end{figure}

In the following sections we apply Theorem~\ref{thm:limits-Hpq} to several cases of interest, including the classical two-dimensional geometries,  and three-dimensional hyperbolic and AdS geometry. Along the way we will discuss further the geometry of those $(\XX(\pfbeta), \Isom(\pfbeta))$ which arise as limits in these cases.

\subsection{The classical two-dimensional geometries}

The two-dimensional Riemannian model geometries of constant curvature may each be realized as subgeometries of $(\RP^2, \PGL_3 \RR)$. In fact, each is defined by a partial flag of quadratic forms (Section~\ref{sec:pfqf}). We use the notation of Section~\ref{sec:Hpq}, and for brevity we will only refer to the space $X$ of the geometry $(X,G)$ when the group $G$ is clear from context:
\begin{itemize}
\item Spherical geometry is (the double cover of)  $\XX(3,0)$.
\item Hyperbolic geometry is $\XX(1,2)$.
\item Euclidean geometry is $\XX((1,0)(2))$
\end{itemize}

The following chart depicts all possible limits of geometries given by a partial flag of quadratic forms in dimension two. 
The completeness/accuracy of the chart is easy to verify using the calculus of Theorem~\ref{thm:limits-Xbeta}.
\begin{displaymath}
\xymatrix{ \XX(3,0) \ar[d] \ar[drr] & \XX(1,2) \ar[dl]\ar[d] & \XX(2,1)\ar[dl]\ar[d]\ar[dr] & \\ \XX((1,0)(2))\ar[dr] & \XX((1,1)(1))\ar[d] & \XX((2,0)(1))\ar[dl] & \XX((1,0)(1,1))\ar[dll] \\  & \XX((1,0)(1)(1)) & & }
\end{displaymath}

The limits of spherical, hyperbolic, and Euclidean geometry may be read off from the chart:
\begin{theorem}
The limits of spherical, hyperbolic, and Euclidean geometry, considered as sub-geometries of projective geometry are the following PFQF geometries:
\begin{itemize}
\item limits of spherical: $\XX(3,0)$ (no degeneration), $\XX((1,0)(2)) = $ Euclidean, and $\XX((1,0)(1)(1)$.
\item limits of hyperbolic: $\XX(1,2)$ (no degeneration), $\XX((1,0)(2)) = $ Euclidean, and $\XX((1,0)(1)(1))$.
\item limits of Eucldiean: $\XX((1,0)(2))$ (no degeneration), and $\XX((1,0)(1)(1))$.
\end{itemize}
\end{theorem}

We give a brief description of the most degenerate two-dimensional partial flag of quadratic forms geometry $\XX((1,0)(1)(1))$ which is a limit of all three classical two-dimensional geometries.
First, the group $$\OO((1,0)(1)(1)) = \begin{pmatrix} \pm 1 & 0 & 0\\ * & \pm 1 & 0 \\ * & * & \pm 1 \end{pmatrix}$$
preserves a full flag $\RR^3 = V_0 \supset V_1 \supset V_2 \supset V_3 = \{0\}$  in $\RR^{3}$. In this case, the space $\XX((1,0)(1)(1)) \cong \mathbb A^2$ is an affine plane, though note that $\OO((1,0)(1)(2))$ is not the full group of affine transformations. The iterated affine bundle structure is:

\begin{displaymath}
    \xymatrix{    \mathbb A^{1,0} \ar[r] & \XX((1,0)(1)(1)) \cong \mathbb A^2 \ar[d]^{\pi_2} \\
              \mathbb A^{1,0} \ar[r] & \XX((1,0)(1)) = \mathbb A^1\ar[d]^{\pi_1} \\
                 & \XX(1,0) = \{pt\} }
\end{displaymath}
Hence $\XX((1,0)(1)(1))$ is an affine two-space, equipped with a translation invariant fibration in Euclidean lines $\mathbb A^{1,0}$ over a base which is also a Euclidean line $\mathbb A^{1,0}$. The group $\OO((1,0)(1)(1))$ is the group of affine transformations which preserve the fibration as well as the metric on the fibers and the base. In the standard basis, $\XX((1,0)(1)(1)) = \{ x_1 \neq 0\}/\RR^*$, which we identify with the affine plane $x_1 = 1$ in $V_0 = \RR^3$. The lines of the foliation are the lines of constant $x_2$ and the (square of the) affine metric on these lines is given by:
$$\varrho^2_2\left(\begin{pmatrix} 1 \\ x_2 \\ x_3 \end{pmatrix} , \begin{pmatrix} 1 \\ x_2 \\ x_3' \end{pmatrix} \right) = (x_3 - x_3')^2.$$

\subsection{Limits of three-dimensional hyperbolic geometry}

 Theorem~\ref{thm:limits-Hpq} gives the limits of three-dimensional hyperbolic geometry $\HH^3 = \XX(1,3)$ as a subgeometry of projective geometry. The results are summarized in the following diagram:
 \begin{equation}
\label{list-H3}
\xymatrix{ & \XX(1,3)\ar[dl]\ar[d]\ar[dr] & \\ \XX((1,0)(3))\ar[d]\ar[dr] & \XX((1,2)(1))\ar[d]\ar[dr] & \XX((1,1)(2))\ar[dll]\ar[d] \\ \XX((1,0)(1)(2))\ar[dr] & \XX((1,0)(2)(1))\ar[d] & \XX((1,1)(1)(1))\ar[dl]  \\ & \XX((1,0)(1)(1)(1)) &}
\end{equation}

We now describe some of the geometries appearing in this list, and their relationships to the Thurston geometries.

\vskip 0.2cm

The geometry $\XX((1,0)(3))$ is Euclidean geometry. 

\vskip 0.2cm

The geometry $\XX((1,2)(1))$ is \emph{half-pipe} geometry, defined by Danciger in \cite{danciger1} and used to construct examples of geometric structures (cone-manifolds) transitioning from hyperbolic geometry to anti de Sitter (AdS) geometry. That $\XX((1,2)(1))$ is a limit of three dimensional AdS geometry also follows from Theorem~\ref{thm:limits-Hpq}, since the projective model for $\operatorname{AdS}^3$ is, in our terminology, $\XX(2,2)$.

\vskip 0.2cm
Next, consider the geometry $\XX((1,0)(2)(1))$, which is also a limit of spherical geometry. The iterated affine bundle structure is:
\begin{displaymath}
    \xymatrix{    \mathbb E^{1} \ar[r] & \XX((1,0)(2)(1)) \cong \mathbb A^3 \ar[d]^{\pi_2} \\
              \mathbb E^{2} \ar[r] & \XX((1,0)(2)) \cong \mathbb E^{2} \ar[d]^{\pi_1} \\
                 & \XX(1,0) = \{*\} }
\end{displaymath}
Hence, $\XX((1,0)(2)(1)$ fibers in Euclidean lines over the Euclidean plane. 
Let $(1\ x\ y\ z)^T$ be coordinates for $\XX((1,0)(2)(1)) = \mathbb A^3$, let $(1\ x'\ y')^T$ be coordinates for $\XX((1,0)(2)) = \mathbb E^2$, and let the projection $\pi_2$ be given by $x' = x, y' = y$.
Then, consider the contact form $\alpha$ on $\mathbb A^3$ defined by 
$$\alpha = dz +x dy - y dx.$$
Note that $d \alpha = 2 \ \pi_2^* dA$, where $dA$ is the area form on the Euclidean plane $\mathbb E^{2}$. Consider the following Riemannian metric $g_{\mathrm{Nil}}$ on $\XX((1,0)(2)(1))$: The fibers of $\pi_2$ are defined to be $g_{\mathrm{Nil}}$ orthogonal to $\ker \alpha$, and $g_{\mathrm{Nil}}$ is defined to be the pull-back by $\pi_2$ of the Euclidean metric on $\ker \alpha$, while $g _{\mathrm{Nil}}(X,X) = \alpha(X)^2$ for $X$ tangent to the $\pi_2$ fiber direction. The Riemannian metric $g_{\mathrm{Nil}}$ makes $\XX((1,0)(2)(1))$ into the model space for \emph{Nil geometry}. Of course, $\OO((1,0)(2)(1))$ does not preserve $\alpha$, nor the metric $g_{\mathrm{Nil}}$. However one may check that the isometries $\Isom(g_{\mathrm{Nil}})$ are a proper subgroup (up to finite index) of $\OO((1,0)(2)(1))$, so Nil geometry locally embeds into  $\XX((1,0)(2)(1))$. 
In coordinates,
$$\Isom_0(g_{\mathrm{Nil}}) = \begin{pmatrix} \pm 1 &  & \\  &  {\OO(2)} & \\  &  & \pm 1\end{pmatrix} \ltimes \begin{pmatrix} 1 & & & \\ a & 1 & & \\ b & 0 & 1 & \\ c & b & -a & 1\end{pmatrix} \subset \OO((1,0)(2)(1)),$$
where $a,b,c \in \RR$ are arbitrary numbers. 
That Nil geometry appears, in this context, as a (sub-geometry of a) limit of hyperbolic geometry is not surprising. Porti \cite{Porti-02} proved that a Nil orbifold with ramification locus transverse to the $\pi_2$ fibration is (metrically) the limit of collapsing hyperbolic cone-manifold structures (after appropriate modification of the collapsing metric).

\vskip 0.2cm
Next consider the geometry $\XX((1,1)(2))$. The iterated affine bundle structure is just one bundle:
\begin{displaymath}
    \xymatrix{    \mathbb E^{2} \ar[r] & \XX((1,1)(2)) \ar[d]^{\pi_1} \\
               & \XX((1,1)) \cong \mathbb H^{1}.  }
\end{displaymath}
In a basis that respects the partial flag, the structure group has the form 
$$\OO((1,1)(2)) = \begin{pmatrix} \OO(1,1)  & \\  &  \OO(2) \end{pmatrix} \ltimes \begin{pmatrix} I_2 & \\ \begin{matrix} * & *\\ * & * \end{matrix} & I_2 \end{pmatrix},$$
where $I_2$ is the $2 \times 2$ identity matrix. In fact, there is a copy of the group $\mathrm{Sol} \cong \SO(1,1) \ltimes \RR^2$ inside of $\OO((1,1)(2))$, which is described in coordinates as follows:
$$\mathrm{Sol} = \begin{pmatrix} \begin{matrix} \cosh z & \sinh z\\ \sinh z & \cosh z \end{matrix}  & \\  &  I_2 \end{pmatrix} \ltimes \begin{pmatrix} I_2 & \\ \begin{matrix} x & x\\ y & -y \end{matrix} & I_2 \end{pmatrix}.$$
In fact, $\mathrm{Sol}$ acts simply transitively on $\XX((1,1)(2))$.  Hence $\XX((1,1)(2))$ is a model for Sol geometry. Four of the eight components of $\Isom \mathrm{Sol}$ lie inside $\OO((1,1)(2))$, corresponding to multiplying the diagonal blocks of $\mathrm{Sol}$ by $\pm 1$. The missing four components are achieved by adding the block diagonal matrix $\operatorname{diag}\left( \begin{pmatrix} 0 & 1\\ 1 & 0 \end{pmatrix}, I_2 \right)$. We may also embed $\mathrm{Sol}$ geometry, in the exact same way, as a sub-geometry of $\XX((1,1)(1)(1))$ (see Figure~\ref{fig:iterated-bundle} for an illustration of the iterated affine bundle structure of $\XX((1,1)(1)(1))$). Let $M$ be a torus bundle over the circle with Anosov mondromy.  In~\cite{HPS}, Huesener--Porti--Su\'arez showed that the natural Sol geometry structure on $M$ is realized as a limit of hyperbolic cone manifold structures, in the sense that the hyperbolic metrics converge to the Sol metric after appropriate (non-isotropic) modification.  It is possible to recast their construction in the context of projective geometry using the theory developed here. Recently, Kozai \cite{Kozai-13} used this projective geometry approach to generalize the work of Huesener--Porti--Su\'arez to the setting of three-manifolds $M$ which fiber as a surface bundle over the circle. Kozai shows (under some assumptions) that the natural singular Sol structure on $M$ may be deformed to nearby singular hyperbolic structures by first deforming from Sol to half-pipe geometry $\XX((1,2)(1))$ and then from half-pipe to hyperbolic geometry.

\subsection{Thurston geometries as limits of hyperbolic geometry}

All eight Thurston geometries locally embed as sub-geometries of real projective geometry (in fact, each embeds up to finite index and coverings). We have now demonstrated that Euclidean geometry is a limit and both $\mathrm{Nil}$ geometry and $\mathrm{Sol}$ geometry locally embed in limits of hyperbolic geometry. We now prove Theorem~\ref{cor:Thurston-geoms_intro}, which says that these are the only Thurston geometries that appear in this way.

\begin{proof}[Proof of Theorem~\ref{cor:Thurston-geoms_intro}]
We show that the projective geometry realizations of the four remaining Thurston geometries, which are $\mathbb S^3$ , $\HH^2 \times \RR$, $\widetilde{\SL_2\RR}$, and $\mathbb S^2 \times \RR$, do not locally the embed in any of the geometries listed in~(\ref{list-H3}). 

Consider spherical geometry $\mathbb S^3$. Up to conjugacy, the local embedding of spherical geometry of projective geometry is unique. It is clear that the structure group $\OO(4)$ does not locally embed as a subgroup in any of the partial flag isometry groups for the geometries appearing in~(\ref{list-H3}).

Next, consider $\mathbb S^2 \times \RR$. The isometry group is (up to finite index) a product $\SO(3) \times \RR$. The only geometry in the list ~(\ref{list-H3}) whose structure group contains a subgroup locally isomorphic to $\SO(3)$ is Euclidean geometry $\XX((1,0)(3))$. Of course $\mathbb S^2 \times \RR$ does not locally embed in Euclidean geometry.

The geometry $\widetilde{\SL_2 \RR}$ is locally isomorphic to (in fact an infinite cyclic cover of) the following subgeometry of projective geometry. The space $\PSL(2,\RR)$ embeds in $\RP^3$ by considering the entries of a $2 \times 2$ matrix as coordinates, and the identity component of the isometry group is given by the linear action of $\PSL_2 \RR \times \PSO(2)$ where the $\PSL_2 \RR$ factor acts on the left and the $\PSO(2)$ factor acts on the right. We show that even the stiffening of this geometry obtained by restricting the structure group to the subgroup $\PSL(2,\RR)$ (acting on the left) does not locally embed in any limit of hyperbolic geometry. For, the only geometry appearing in~(\ref{list-H3}) whose isometry group contains a subgroup locally isomorphic to $\PSL(2,\RR) \cong \SO_0(2,1)$ is half-pipe geometry $\XX((1,2)(1))$. Any such subgroup is conjugate into the block diagonal subgroup $\operatorname{diag}(\OO(2,1), 1)$ and it is then easy to see that such a subgroup does not act transitively on $\XX((1,2)(1))$, but rather preserves a totally geodesic subspace (a copy of the hyperbolic plane, see \cite{danciger1}). Therefore, since the left action of $\PSL_2R$ is transitive, we have shown that $\widetilde{\SL_2 \RR}$ does not locally embed in half-pipe geometry.

Finally, consider $\HH^2 \times \RR$ geometry. The isometry group is (up to finite index) the product $\SO(2,1) \times \RR$. Again, the only limit of hyperbolic geometry whose isometry group contains a subgroup $H$ locally isomorphic to $\SO(2,1)$ is $\XX((1,2)(1))$, half-pipe geometry. However, the centralizer of the smallest such subgroup $H = \SO_0(2,1) \times \{1\} $ in $\OO((1,2)(1))$ is $\operatorname{diag}(\pm I_3, \pm 1)$, so in particular there is no subgroup locally isomorphic to $\SO(2,1) \times \RR$ inside of $\OO((1,2)(1))$.

\end{proof}

\small
\bibliography{refs} 

\begin{thebibliography}{10}

\bibitem{MB}
M.~Berger.
\newblock Les espaces sym\'etriques noncompacts.
\newblock {\em Ann. Sci. \'Ecole Norm. Sup. (3)}, 74:85--177, 1957.

\bibitem{bishop}
E.~Bishop.
\newblock Conditions for the analyticity of certain sets.
\newblock {\em Michigan Math. J.}, 11:289--304, 1964.

\bibitem{Boileau_Leeb_Porti_orbi}
M.~Boileau, B.~Leeb, and J.~Porti.
\newblock Geometrization of 3-dimensional orbifolds.
\newblock {\em Ann. of Math. (2)}, 162(1):195--290, 2005.

\bibitem{bridson}
M.~R. Bridson, P.~de~la Harpe, and V.~Kleptsyn.
\newblock The {C}habauty space of closed subgroups of the three-dimensional
  {H}eisenberg group.
\newblock {\em Pacific J. Math.}, 240(1):1--48, 2009.

\bibitem{burde}
D.~Burde.
\newblock Contractions of {L}ie algebras and algebraic groups.
\newblock {\em Arch. Math. (Brno)}, 43(5):321--332, 2007.

\bibitem{chabauty}
C.~Chabauty.
\newblock Limite d'ensembles et g\'eom\'etrie des nombres.
\newblock {\em Bull. Soc. Math. France}, 78:143--151, 1950.

\bibitem{Choi_Goldman}
S.~Choi and W.~Goldman.
\newblock Topological tameness of margulis spacetimes.
\newblock Preprint, http://lanl.arxiv.org/abs/1204.5308, 2012.

\bibitem{delp}
D.~Cooper and K.~Delp.
\newblock The marked length spectrum of a projective manifold or orbifold.
\newblock {\em Proc. Amer. Math. Soc.}, 138(9):3361--3376, 2010.

\bibitem{Cooper_Hodgson_Kerckhoff}
D.~Cooper, C.~D. Hodgson, and S.~P. Kerckhoff.
\newblock {\em Three-dimensional orbifolds and cone-manifolds}, volume~5 of
  {\em MSJ Memoirs}.
\newblock Mathematical Society of Japan, Tokyo, 2000.
\newblock With a postface by Sadayoshi Kojima.

\bibitem{danciger1}
J.~{Danciger}.
\newblock {A Geometric transition from hyperbolic to anti de Sitter geometry}.
\newblock {\em Geom. Topol.}, 17(5):3077--3134, 2013.

\bibitem{danciger2}
J.~{Danciger}.
\newblock {Ideal triangulations and geometric transitions}.
\newblock {\em J. Topol.}, to appear, 2014.

\bibitem{Danciger_Gueritaud_Kassel_topology}
J.~Danciger, F.~Gueritaud, and F.~Kassel.
\newblock Geometry and topology of complete lorentz spacetimes of constant
  curvature.
\newblock Preprint, http://arxiv.org/abs/1306.2240, 2013.

\bibitem{Danciger_Maloni_Schlenker}
J.~Danciger, S.~Maloni, and J.-M. Schlenker.
\newblock Polyhedra inscribed in a quadric.
\newblock in preparation, 2014.

\bibitem{delaharpe}
P.~{de la Harpe}.
\newblock {Spaces of closed subgroups of locally compact groups}.
\newblock {\em ArXiv e-prints}, July 2008.

\bibitem{Gorodnik_Oh_Shah}
A.~Gorodnik, H.~Oh, and N.~Shah.
\newblock Integral points on symmetric varieties and {S}atake
  compactifications.
\newblock {\em Amer. J. Math.}, 131(1):1--57, 2009.

\bibitem{GJT}
Y.~Guivarc'h, L.~Ji, and J.~C. Taylor.
\newblock {\em Compactifications of symmetric spaces}, volume 156 of {\em
  Progress in Mathematics}.
\newblock Birkh\"auser Boston Inc., Boston, MA, 1998.

\bibitem{Haettel2}
T.~Haettel.
\newblock Compactification de {C}habauty des espaces sym\'etriques de type non
  compact.
\newblock {\em J. Lie Theory}, 20(3):437--468, 2010.

\bibitem{Schlichtkrull}
G.~Heckman and H.~Schlichtkrull.
\newblock {\em Harmonic analysis and special functions on symmetric spaces},
  volume~16 of {\em Perspectives in Mathematics}.
\newblock Academic Press Inc., San Diego, CA, 1994.

\bibitem{Hodgson}
C.~D. Hodgson.
\newblock Degeneration and regeneration of hyperbolic structures on three-
  manifolds.
\newblock PhD thesis, Princeton University, 1986.

\bibitem{HPS}
M.~Huesener, J.~Porti, and E.~Su\'arez.
\newblock Regenerating singular hyperbolic structures from sol.
\newblock {\em J. Diff. Geom.}, 59:439--478, 2001.

\bibitem{Wigner}
E.~Inonu and E.~P. Wigner.
\newblock On the contraction of groups and their representations.
\newblock {\em Proc. Nat. Acad. Sci. U. S. A.}, 39:510--524, 1953.

\bibitem{Kozai-13}
K.~Kozai.
\newblock {\em Singular hyperbolic structures on pseudo-Anosov mapping tori}.
\newblock Ph.D. thesis, Stanford University, 2013.

\bibitem{molnar}
E.~Moln{\'a}r.
\newblock The projective interpretation of the eight {$3$}-dimensional
  homogeneous geometries.
\newblock {\em Beitr\"age Algebra Geom.}, 38(2):261--288, 1997.

\bibitem{nomizu}
K.~Nomizu.
\newblock Invariant affine connections on homogeneous spaces.
\newblock {\em Amer. J. Math.}, 76:33--65, 1954.

\bibitem{Porti-98}
J.~Porti.
\newblock Regenerating hyperbolic and spherical cone structures from euclidean
  ones.
\newblock {\em Topology}, 37(2):365--392, 1998.

\bibitem{Porti-02}
J.~Porti.
\newblock Regenerating hyperbolic cone structures from nil.
\newblock {\em Geom. Topol.}, 6:815--852, 2002.

\bibitem{Rossmann}
W.~Rossmann.
\newblock The structure of semisimple symmetric spaces.
\newblock {\em Canad. J. Math.}, 31(1):157--180, 1979.

\bibitem{thiel}
B.~Thiel.
\newblock {\em Einheitliche Beschreibung der acht Thurstonschen Geometrien}.
\newblock Diplomarbeit, Universitat zu Gottingen, 1997.

\bibitem{Tworzewski}
P.~Tworzewski and T.~Winiarski.
\newblock Limits of algebraic sets of bounded degree.
\newblock {\em Univ. Iagel. Acta Math.}, XXIV(24):151--153, 1984.

\end{thebibliography}
\bibliographystyle{abbrv} 

\end{document}